\documentclass[12pt]{amsart}
\usepackage[english]{babel}
\usepackage[letterpaper,top=3cm,bottom=3cm,left=3.25cm,right=3.25cm,marginparwidth=1.75cm]{geometry}
\usepackage{amsmath}
\numberwithin{equation}{section}
\usepackage{amsthm}
\usepackage{amssymb}
\usepackage{mathtools}
\usepackage{graphicx}
\usepackage{algorithmic}
\usepackage{tikz}
\usetikzlibrary{arrows.meta,calc,decorations.pathmorphing}
\usepackage{float}
\usepackage{caption}
\DeclareCaptionLabelFormat{andtable}{#1~#2  \&  \tablename~\thetable}
\usepackage{subcaption}
\usepackage{booktabs}
\usepackage{multirow}
\usepackage[colorlinks=true, linkcolor=blue, citecolor=green, urlcolor=green]{hyperref}
\newcommand{\cc}{\mathrel{\subset\subset}}

\definecolor{waterFill}{RGB}{120,185,240}
\definecolor{waterEdge}{RGB}{18,94,185}
\definecolor{anodeBase}{RGB}{205,117,52}
\definecolor{anodeTop}{RGB}{219,151,95}
\definecolor{anodeDark}{RGB}{112,52,22}
\definecolor{cathodeBase}{RGB}{160,164,164}
\definecolor{cathodeDark}{RGB}{88,92,92}
\definecolor{ionOrange}{RGB}{196,86,0}

\newenvironment{paperfigure}[1][htbp]{\begin{figure}[#1]}{\end{figure}}
\newcounter{proofstep}
\newenvironment{step}[1]{%
  \stepcounter{proofstep}%
  \par\medskip\noindent\textbf{Step \theproofstep} (#1).\quad\ignorespaces%
}{\par\medskip}

\newtheorem{theoremp}{Theorem}[section]
\newenvironment{theorem}{\vspace{5pt}\begin{theoremp}}{\end{theoremp}\vspace{5pt}}
\newtheorem{propositionp}[theoremp]{Proposition}

\newtheorem{corollaryp}{Corollary}[section]
\newenvironment{corollary}{\vspace{5pt}\begin{corollaryp}}{\end{corollaryp}\vspace{5pt}}
\newtheorem{examplep}[theoremp]{Example}

\newtheorem{lemmap}[theoremp]{Lemma}
\newenvironment{lemma}{\vspace{5pt}\begin{lemmap}}{\end{lemmap}\vspace{5pt}}

\makeatletter
\newtheoremstyle{boldremark}% <name>
  {5pt}% <Space above>
  {5pt}% <Space below>
  {\normalfont}% <Body font>
  {}% <Indent amount>
  {\bfseries}% <Theorem head font> — changed from \itshape to \bfseries
  {.}% <Punctuation after theorem head>
  { }% <Space after theorem head>
  {}% <Theorem head spec>
\makeatother
\theoremstyle{boldremark}
\newtheorem{remark}{Remark}[section]

\makeatletter
\def\@adminfootnotes{%
  \let\@makefnmark\relax \let\@thefnmark\relax
  \ifx\@empty\@date\else \@footnotetext{\@setdate}\fi
  \ifx\@empty\thankses\else \@footnotetext{%
    \def\par{\let\par\@par}\@setthanks}%
  \fi
  \ifx\@empty\@subjclass\else \@footnotetext{\@setsubjclass}\fi
  \ifx\@empty\@keywords\else \@footnotetext{\@setkeywords}\fi
}
\makeatother

\title[Singular limit phenomenon in electrochemistry]{Singular limit phenomenon in a nonlinear elliptic model arising in electrochemistry}
\author[Daniel Fern\'andez]{Daniel Fern\'andez\textsuperscript{$\ast$}}
\thanks{\textsuperscript{$\ast$}Chair for Dynamics, Control, Machine Learning, and Numerics (Alexander von Humboldt Professorship), Department of Mathematics, Friedrich-Alexander-Universit\"at Erlangen-N\"urnberg, 91058 Erlangen, Germany.\newline
\emph{Email:} \texttt{daniel.fernandez@fau.de}.}

\begin{document}

\begin{abstract}
We study a singularly perturbed harmonic problem, with nonlinear Neumann
boundary conditions of signed exponential type, arising in galvanic corrosion applications. We focus on the low-conductivity regime, which leads to a singular limit in which the solution presents a jump in the boundary. We prove convergence of the solution traces in the corresponding fractional Sobolev spaces, identify
the interior limit solution, and obtain a sharp logarithmic energy expansion for smooth
boundary junctions. The proofs rely mostly on trace Moser--Trudinger estimates for showing well-posedness and regularity, and on mollifier approximation techniques for the singular limit analysis. Numerical experiments test the predicted
convergence rate.
\end{abstract}

\maketitle

% !TEX root = ../main.tex
\section{Introduction}\label{sec:introduction}

We study the elliptic model that arises in galvanic corrosion when the electrolyte potential is
harmonic in the interior and the metal-electrolyte interface is governed by highly nonlinear Butler--Volmer current laws
\cite{BardFaulkner2001,NewmanThomasAlyea2004,Rubinstein1990}. Related elliptic corrosion models with
nonlinear boundary reactions appear, for instance, in
\cite{VogeliusXu1998,KavianVogelius2003,BhatMoskow2006}.

The unknown
$\phi_\kappa$ denotes the steady electric potential in a bounded electrolyte domain
$\Omega\subset\mathbb{R}^2$. The boundary is decomposed into two reaction boundary parts $\Gamma_a$, $\Gamma_c$, and
a nonempty zero-flux part $\Gamma_N$, and $\phi_\kappa$ solves
\begin{equation}\label{eq:intro_model_pde}
\begin{cases}
-\Delta \phi_\kappa = 0 & \text{in }\Omega,\\
\partial_\nu \phi_\kappa = 0 & \text{on }\Gamma_N,\\
\partial_\nu \phi_\kappa = -i_a(\phi_\kappa) / \kappa & \text{on }\Gamma_a,\\
\partial_\nu \phi_\kappa = -i_c(\phi_\kappa) / \kappa & \text{on }\Gamma_c,
\end{cases}
\end{equation}
where $\nu$ is the exterior unit normal and $\kappa>0$ is the electrolyte conductivity. The functions $i_a$ and $i_c$
are signed exponentials with distinct zeros $\phi_a$ and $\phi_c$. In corrosion terminology, $\Gamma_a$ and
$\Gamma_c$ are anodic and cathodic portions of the exposed interface, while $\Gamma_N$ represents an insulated or
nonreacting boundary part.

This model is meant to isolate the activation-controlled part of a galvanic cell in a regime where the electrolyte
conductivity is small relative to the interfacial exchange-current scale. Equivalently, after nondimensionalization,
small $\kappa$ may represent a poorly conducting electrolyte, a long current path, or comparatively fast boundary
charge-transfer kinetics. The analysis does not include concentration depletion, space-charge layers, evolving
corrosion fronts, or changes in which boundary portions are anodic and cathodic. Those effects are important in full
corrosion simulations \cite{femreview,volume,fem1,per,phase}, but here they are deliberately frozen in order to
identify the singular elliptic mechanism created by incompatible cathodic and anodic equilibrium potentials.

The main difficulty is the \emph{boundary singularity} generated by the cathode-anode boundary decomposition. In two
dimensions the exponential Neumann law lies in the \emph{trace-critical regime} governed by trace Moser--Trudinger
estimates \cite{Cianchi2008,LiLiu2005,Yang2006TraceTM,Yang2007SharpTraceTM}.

As $\kappa\to0^+$, the prefactor
$1/\kappa$ forces the trace on $\Gamma_a\cup\Gamma_c$ toward the two equilibrium values. If these limiting values
meet at finitely many points of $\overline{\Gamma_a}\cap\overline{\Gamma_c}$, the limiting boundary datum is
discontinuous, and the critical $H^{1/2}$ cost of approximating this jump produces \emph{logarithmic energy growth}
\cite{Mazya2011,hitch}.

The central result of the paper is the \emph{exact leading coefficient} of this growth.
The proof isolates the three ingredients behind this phenomenon: a stiff boundary primitive that selects the
equilibrium trace, a linear elliptic extension operator that propagates the trace into the interior, and the local
borderline capacity of a boundary step.

\subsection{Contributions}

The main results identify the singular limit generated by the reactive boundary structure.
\begin{itemize}
    \item \textbf{Boundary convergence and mixed harmonic selection.}
    Under the standing geometric assumptions \eqref{eq:boundary_decomp}--\eqref{eq:separation_condition}, the solution traces $\operatorname{Tr}\phi_\kappa$ from equation \eqref{eq:intro_model_pde} converge on the reaction boundary $\Gamma_c\cup\Gamma_a$ to the \emph{equilibrium profile}
    \begin{equation}\label{eq:equilibrium_profile}
    \Phi_0:=\phi_c\mathbf{1}_{\Gamma_c}+\phi_a\mathbf{1}_{\Gamma_a}
    \quad\text{on }\Gamma_c\cup\Gamma_a,
    \end{equation}
    and satisfy, for every $0\le s<1/2$, and sufficiently small $\kappa>0$,
    \[
    \|\operatorname{Tr}\phi_\kappa-\Phi_0\|_{H^s(\Gamma_c\cup\Gamma_a)}
    \le C_s\,\kappa^{1/2-s}|\log\kappa|^{1/2}.
    \]
    The solutions converge locally in the interior to the bounded mixed harmonic function selected by
    $\Phi_0$, with the quantitative interior rate inherited from the boundary estimate
    (Theorem~\ref{thm:boundary_harmonic_selection} and Corollary~\ref{cor:interior_selection}).

    \item \textbf{Exact smooth boundary-jump constant.}
    Let $J_\kappa$ be the energy functional associated with problem \eqref{eq:intro_model_pde}, and assume $\overline{\Gamma_a}\cap\overline{\Gamma_c}$ is a set of $N$ junction points. Then,
    \[
    J_\kappa(\phi_\kappa)
    =\frac{N(\phi_c-\phi_a)^2}{2\pi}|\log\kappa|
    +O(\log|\log\kappa|).
    \]

    Thus the low-conductivity limit is singular: the energy cannot remain bounded when adjacent boundary pieces try to
    impose two different equilibrium potentials. The solutions still select the cathodic and anodic equilibrium values
    away from the junctions, but the transition between them is compressed into boundary layers whose total leading
    cost is exactly the logarithmic term above (Theorem~\ref{thm:energy_scale}).

    \item \textbf{Numerical validation and mesh design.}
    The finite-element experiments focus on theorem-facing diagnostics: boundary and interior convergence, normalized
    energy, exact-constant refinement, mesh refinement guided by the boundary-layer scale near each jump, and corrugated
    tests with several jump points, using standard conforming finite-element discretizations. The same local balance
    identifies the heuristic \emph{reaction-layer mesh scale} $h_{\min}\simeq\kappa$ near each smooth jump point
    (Remark~\ref{rem:mesh_scale}).
\end{itemize}

\subsection{Related work}

\noindent We briefly situate the paper relative to corrosion models, exponential Neumann singular limits, and the
functional-analytic tools used below.

\noindent\textbf{Corrosion and nonlinear boundary models.}
The signed exponential boundary law is the Butler--Volmer current-overpotential relation
\cite{BardFaulkner2001,NewmanThomasAlyea2004}. Related nonlinear elliptic corrosion models address
well-posedness, singular exact solutions, boundary blow-up, and heterogeneous or periodic reactions
\cite{VogeliusXu1998,KavianVogelius2003,BhatMoskow2006}.
Those works provide the electrochemical and elliptic background, but not the low-conductivity mechanism in which a
stiff boundary energy selects incompatible adjacent equilibrium traces.

\noindent\textbf{Boundary concentration limits.}
Singular limits for exponential Neumann data often produce boundary concentration: large-amplitude solutions and
rescaled fluxes converging to Dirac masses
\cite{DavilaDelPinoMussoWei2006,GuoLiu2006}, or, in Steklov-type
variants, concentration along boundary geodesics \cite{PaganiPierottiPistoiaVaira2016}. Here the range bound prevents
blow-up, and the singular object is instead a discontinuous equilibrium trace. The leading cost is therefore a
fractional boundary-jump cost, not a concentration energy.

\noindent\textbf{Existence and regularity of the solution.}
The fixed-conductivity problem relies on trace Moser--Trudinger control for exponential boundary growth
\cite{Cianchi2008,Yang2007SharpTraceTM}. The singular limiting arguments also use mixed harmonic extension
theory \cite{Kenig1994,TaylorOttBrown2013} and standard Neumann shift results
\cite{Grisvard2011}. These tools supply compactness, coercivity, and interior smoothing, but not the
sharp low-conductivity energy constant.

\noindent\textbf{Boundary-jump cost.}
The local cost of smoothing a one-dimensional step is classical in fractional Sobolev theory and potential theory
\cite{Mazya2011,hitch}, and related logarithmic energies occur in Ginzburg--Landau settings
\cite{SandierSerfaty2007}. The contribution here is to show that this jump cost is
selected by the Butler--Volmer boundary functional and to compute its smooth-junction leading constant.

\subsection{Outline of the paper}

Section~\ref{sec:main-results} states the model and the main singular-limit results. Section~\ref{sec:preliminary-lemmas}
collects the preliminary estimates, and Section~\ref{sec:proofs} gives the proofs of the main theorems.
Section~\ref{sec:numerics} gives the numerical validation, and Section~\ref{sec:conclusion} summarizes the conclusions.
 % !TEX root = ../main.tex
\section{Main results}\label{sec:main-results}

This section fixes the mathematical setting and states the two singular-limit results proved later. The main point is
that the small conductivity parameter does not merely produce a large boundary reaction; it selects a discontinuous
equilibrium trace on the reactive boundary, and the cost of resolving the resulting cathode-anode jumps determines the
leading asymptotic energy.

\subsection{Model setup and geometric framework}\label{subsec:model-setup}

Throughout, $\Omega\subset\mathbb{R}^2$ is a bounded, connected domain that is either $C^{1,1}$ or a Lipschitz
curvilinear polygon. In the latter case, $\partial\Omega$ is a finite union of $C^{1,1}$ arcs meeting at a finite
corner set $\mathcal C_\Omega$, and every interior angle lies in $(0,2\pi)$.
The boundary is partitioned into cathodic, anodic, and insulated arcs:
\begin{equation}\label{eq:boundary_decomp}
\partial\Omega=\overline{\Gamma_c}\cup\overline{\Gamma_a}\cup\overline{\Gamma_N},
\end{equation}
where $\Gamma_c,\Gamma_a,\Gamma_N\subset\partial\Omega$ are pairwise disjoint relatively open finite unions of boundary
arcs, with $|\Gamma_c|,|\Gamma_a|,|\Gamma_N|>0$. The finitely many arc endpoints are
ignored in trace statements. We write the reactive boundary and the cathode-anode junction set as
\begin{equation}\label{eq:reactive_sets}
\Gamma_*:=\Gamma_c\cup\Gamma_a,
\qquad
\Sigma_*:=\overline{\Gamma_c}\cap\overline{\Gamma_a}.
\end{equation}

Here $\Sigma_*$ is the finite cathode-anode junction set. In the curvilinear polygonal case, cathode-anode junctions
are assumed to be smooth boundary points, not corners. These junctions are also assumed to stay away from the
zero-flux boundary:
\begin{equation}\label{eq:separation_condition}
\operatorname{dist}\!\left(\Sigma_*,\mathcal C_\Omega\right)>0,
\qquad
\operatorname{dist}\!\left(\Sigma_*,\overline{\Gamma_N}\right)>0.
\end{equation}
with $\mathcal C_\Omega=\emptyset$ in the $C^{1,1}$ case and the first condition then understood as vacuous.
Thus each cathode-anode junction has a reactive boundary neighborhood, while the global problem contains insulated
arcs or curvilinear corners elsewhere.

In this geometric setting the steady electric potential satisfies equation \eqref{eq:intro_model_pde}, where $\kappa>0$ is the
electrolyte conductivity and the Neumann boundary data are given by the signed exponential currents
\begin{equation}\label{i2}
\begin{cases}
    i_c(\phi)=i_{c,0}\left(\exp(C_1(\phi-\phi_c))-\exp(-C_2(\phi-\phi_c))\right),\\
    i_a(\phi)=i_{a,0}\left(\exp(A_2(\phi-\phi_a))-\exp(-A_1(\phi-\phi_a))\right).
\end{cases}
\end{equation}

Here $i_{c,0},i_{a,0},C_1,C_2,A_1,A_2$ are positive constants and $\phi_c>\phi_a$ are the two equilibrium potentials.
The currents are strictly increasing and vanish at their own equilibria. Consequently, when the prefactor
$1/\kappa$ becomes large, the cathodic boundary energetically favors the value $\phi_c$ and the anodic boundary
favors $\phi_a$. If both types of arcs meet, these two preferred values are incompatible at the junction.
The boundary decomposition and the resulting cathode-anode junction geometry are illustrated in
Figure~\ref{fig:corrosion_illustration}, where the planar setting used in the analysis appears as a cross-section of
the three-dimensional corrosion configuration.

\begin{paperfigure}[t]
\centering
\begin{tikzpicture}[>=Latex]
\node[inner sep=0, anchor=south west] (threeD) at (0,0)
    {\resizebox{0.42\textwidth}{!}{\begin{tikzpicture}[
    x=0.78cm,
    y=0.78cm,
    line cap=round,
    line join=round,
    >=Latex,
    metaledge/.style={draw=black!38,line width=0.55pt},
    waterline/.style={draw=waterEdge,line width=0.9pt},
    transport/.style={waterEdge!90!black,line width=0.8pt,-{Latex[length=2.4mm,width=1.7mm]},opacity=0.88},
    dissolve/.style={ionOrange,line width=0.9pt,-{Latex[length=2.4mm,width=1.8mm]},opacity=0.95},
    pointer/.style={line width=0.7pt},
    lab/.style={font=\Large\rmfamily},
    mathlab/.style={font=\Large\rmfamily}
]

\path[fill=black,opacity=0.08]
  (0.25,-0.25) -- (10.4,-0.25) -- (11.85,0.68) -- (1.65,0.68) -- cycle;

% Gray slicing plane kept behind the object so the 3D scene stays unchanged.
\path[fill=black!14,fill opacity=0.42,even odd rule]
  (-0.95,0.06) rectangle (12.35,6.30)
  (0,0) -- (10,0) -- (11.4,0.95) -- (11.4,5.67)
  .. controls (10.35,5.76) and (9.85,5.45) .. (8.78,5.58)
  .. controls (7.72,5.71) and (7.02,5.42) .. (6.05,5.55)
  .. controls (4.92,5.70) and (4.25,5.42) .. (3.18,5.62)
  .. controls (2.52,5.75) and (1.92,5.76) .. (1.4,5.70)
  .. controls (1.05,5.84) and (0.52,4.91) .. (0,4.75)
  -- cycle;
\draw[black!68,line width=0.85pt] (-0.95,0.06) -- (-0.95,6.30) -- (12.35,6.30) -- (12.35,0.06);
\draw[black!68,line width=0.85pt] (-0.95,0.06) -- (0.00,0.06);
\draw[black!55,line width=0.72pt,dash pattern=on 2.0pt off 1.6pt] (0.00,0.06) -- (11.40,0.06);
\draw[black!68,line width=0.85pt] (11.40,0.06) -- (12.35,0.06);

\path[fill=anodeTop,metaledge]
  (0,1.8) -- (5,1.8) -- (6.4,2.75) -- (1.4,2.75) -- cycle;

\path[fill=cathodeBase!55,metaledge]
  (5,1.8) -- (10,1.8) -- (11.4,2.75) -- (6.4,2.75) -- cycle;

\path[fill=cathodeDark!78,metaledge]
  (10,0) -- (11.4,0.95) -- (11.4,2.75) -- (10,1.8) -- cycle;

\path[fill=anodeBase,metaledge]
  (0,0) -- (5,0) -- (5,1.8) -- (0,1.8) -- cycle;

\path[fill=cathodeBase,metaledge]
  (5,0) -- (10,0) -- (10,1.8) -- (5,1.8) -- cycle;

\path[fill=black,opacity=0.08]
  (10,0) -- (11.4,0.95) -- (11.4,2.75) -- (10,1.8) -- cycle;

\node[font=\huge\rmfamily,anchor=center] at (2.5,0.9) {anode};
\node[font=\huge\rmfamily,anchor=center] at (7.5,0.9) {cathode};

\draw[black!45,line width=0.65pt] (5,0) -- (5,1.8) -- (6.4,2.75);

\path[fill=waterFill,fill opacity=0.28,draw=waterEdge,draw opacity=0.80,line width=0.75pt]
  (0,1.8) -- (1.4,2.75)
  -- (1.4,5.70)
  .. controls (1.05,5.84) and (0.52,4.91) .. (0,4.75)
  -- cycle;

\path[fill=waterFill,fill opacity=0.36,draw=waterEdge,draw opacity=0.82,line width=0.75pt]
  (0,1.8) -- (10,1.8) --
  (10,4.72)
  .. controls (9.15,4.72) and (8.65,4.89) .. (7.75,4.77)
  .. controls (6.65,4.62) and (5.95,4.79) .. (5.10,4.68)
  .. controls (4.18,4.55) and (3.50,4.90) .. (2.55,4.83)
  .. controls (1.55,4.76) and (1.05,5.12) .. (0,4.75)
  -- cycle;

\path[fill=waterFill,fill opacity=0.32,draw=waterEdge,draw opacity=0.83,line width=0.75pt]
  (10,1.8) -- (11.4,2.75) --
  (11.4,5.67)
  .. controls (11.05,5.53) and (10.55,4.86) .. (10,4.72)
  -- cycle;

\path[fill=waterFill!80!white,fill opacity=0.58,draw=waterEdge,draw opacity=0.96,line width=0.95pt]
  (0,4.75)
  .. controls (1.05,5.12) and (1.55,4.76) .. (2.55,4.83)
  .. controls (3.50,4.90) and (4.18,4.55) .. (5.10,4.68)
  .. controls (5.95,4.79) and (6.65,4.62) .. (7.75,4.77)
  .. controls (8.65,4.89) and (9.15,4.72) .. (10,4.72)
  -- (11.4,5.67)
  .. controls (10.35,5.76) and (9.85,5.45) .. (8.78,5.58)
  .. controls (7.72,5.71) and (7.02,5.42) .. (6.05,5.55)
  .. controls (4.92,5.70) and (4.25,5.42) .. (3.18,5.62)
  .. controls (2.52,5.75) and (1.92,5.76) .. (1.4,5.70)
  -- cycle;

\draw[white,opacity=0.55,line width=0.55pt]
  (1.00,4.98) .. controls (1.65,5.17) and (2.15,4.90) .. (2.75,5.05);
\draw[white,opacity=0.50,line width=0.55pt]
  (3.20,5.13) .. controls (4.10,5.33) and (4.95,4.99) .. (5.70,5.16);
\draw[white,opacity=0.45,line width=0.55pt]
  (6.55,5.20) .. controls (7.20,5.02) and (7.80,5.30) .. (8.55,5.14);
\draw[white,opacity=0.45,line width=0.50pt]
  (8.85,5.35) .. controls (9.45,5.52) and (10.15,5.20) .. (10.82,5.42);

\draw[waterline] (0,1.8) -- (5,1.8) -- (10,1.8);
\draw[waterline,opacity=0.75] (10,1.8) -- (11.4,2.75);
\draw[waterline,opacity=0.55] (0,1.8) -- (1.4,2.75);

% Section traces on the fluid part of the cut.
\draw[black,line width=1.20pt,dash pattern=on 2.0pt off 1.6pt] (0.96,2.46) -- (0.96,5.41);
\draw[black,line width=1.25pt] (10.96,0.66) -- (10.96,2.46);
\draw[black,line width=1.25pt] (10.96,2.46) -- (10.96,5.38);
\draw[black,line width=1.25pt]
  (0.96,5.41)
  .. controls (2.01,5.78) and (2.51,5.42) .. (3.51,5.49)
  .. controls (4.46,5.56) and (5.14,5.21) .. (6.06,5.34)
  .. controls (6.91,5.45) and (7.61,5.28) .. (8.71,5.43)
  .. controls (9.61,5.55) and (10.11,5.38) .. (10.96,5.38);

\path[fill=anodeDark,opacity=0.82]
  (0.55,1.83)
  .. controls (0.75,2.12) and (1.08,1.98) .. (1.24,2.18)
  .. controls (1.42,2.42) and (1.73,2.00) .. (1.98,2.25)
  .. controls (2.20,2.47) and (2.52,2.04) .. (2.74,2.22)
  .. controls (3.00,2.45) and (3.45,2.05) .. (3.67,2.20)
  .. controls (3.45,1.96) and (3.08,1.90) .. (2.72,1.94)
  .. controls (2.30,1.98) and (1.98,1.70) .. (1.62,1.86)
  .. controls (1.26,2.03) and (0.88,1.68) .. (0.55,1.83) -- cycle;

\foreach \x/\y/\rx/\ry in {
  0.82/1.96/0.20/0.08,
  1.18/2.10/0.18/0.10,
  1.55/1.98/0.26/0.10,
  1.92/2.20/0.20/0.11,
  2.35/2.06/0.24/0.10,
  2.86/2.18/0.20/0.09,
  3.33/2.04/0.22/0.08
}{
  \path[fill=black,opacity=0.30] (\x,\y) ellipse [x radius=\rx,y radius=\ry];
  \path[fill=orange!35!white,opacity=0.45] ($(\x,\y)+(0.04,0.04)$) ellipse [x radius=0.45*\rx,y radius=0.35*\ry];
}

\draw[dissolve] (1.05,2.08) .. controls (1.10,2.55) and (1.05,2.95) .. (1.10,3.35);
\draw[dissolve] (1.58,2.08) .. controls (1.62,2.68) and (1.58,3.15) .. (1.60,3.78);
\draw[dissolve] (2.10,2.12) .. controls (2.10,2.55) and (2.06,2.92) .. (2.12,3.18);
\draw[dissolve] (2.70,2.12) .. controls (2.72,2.72) and (2.72,3.15) .. (2.76,3.52);

\foreach \x/\y in {0.98/3.70,1.55/4.08,2.22/3.66,2.68/4.03,1.90/3.25}
  \fill[ionOrange] (\x,\y) circle[radius=0.07];

\draw[transport] (3.45,3.85) .. controls (4.90,3.15) and (6.55,3.08) .. (8.20,3.42);
\draw[transport] (3.25,3.15) .. controls (4.75,2.90) and (6.75,2.88) .. (8.78,3.12);
\draw[transport] (3.40,2.48) .. controls (5.05,2.38) and (6.55,2.38) .. (7.88,2.58);

\node[font=\Huge\rmfamily,text=waterEdge!85!black] at (6.28,3.98) {$\Omega$};

\node[mathlab,text=anodeDark!90!black] (Ga) at (3.05,2.93) {$\Gamma_a$};
\draw[anodeDark!90!black,pointer] (Ga.south west) -- (2.72,2.05);
\fill[anodeDark!90!black] (2.72,2.05) circle[radius=0.07];

\node[mathlab,text=cathodeDark!95!black] (Gc) at (8.75,2.75) {$\Gamma_c$};
\draw[cathodeDark!95!black,pointer] (Gc.south west) -- (7.98,1.84);
\fill[cathodeDark!95!black] (7.98,1.84) circle[radius=0.07];

\node[lab,text=black] (GN) at (10.10,6.45) {$\Gamma_N$};
\draw[black,line width=0.85pt,-{Latex[length=3.2mm,width=2.0mm]}] (GN.south west) -- (9.18,5.34);

\end{tikzpicture}
 }};
\node[inner sep=0, anchor=west] (twoD) at ([xshift=1.45cm]threeD.east)
    {\resizebox{0.42\textwidth}{!}{\begin{tikzpicture}[
    x=0.84cm,
    y=0.84cm,
    line cap=round,
    line join=round,
    >=Latex,
    metaledge/.style={draw=black!38,line width=0.55pt},
    pointer/.style={line width=0.7pt},
    lab/.style={font=\Large\rmfamily},
    mathlab/.style={font=\Large\rmfamily}
]

\path[fill=black,opacity=0.07]
  (0.18,-0.18) rectangle (10.18,0.12);

\path[fill=anodeBase,metaledge]
  (0,0) rectangle (5,1.15);
\path[fill=cathodeBase,metaledge]
  (5,0) rectangle (10,1.15);
\draw[black!45,line width=0.6pt] (5,0) -- (5,1.15);

\path[fill=waterFill,fill opacity=0.38,draw=waterEdge,draw opacity=0.92,line width=0.9pt]
  (0,1.15) -- (10,1.15) -- (10,4.16)
  .. controls (9.15,4.18) and (8.55,4.36) .. (7.62,4.23)
  .. controls (6.68,4.10) and (5.92,4.33) .. (5.00,4.20)
  .. controls (4.06,4.07) and (3.28,4.34) .. (2.30,4.24)
  .. controls (1.44,4.15) and (0.74,4.32) .. (0,4.16)
  -- cycle;

\draw[anodeDark!90!black,line width=2.0pt] (0,1.15) -- (5,1.15);
\draw[cathodeDark!95!black,line width=2.0pt] (5,1.15) -- (10,1.15);

\draw[waterEdge,line width=0.95pt]
  (0,4.16)
  .. controls (0.74,4.32) and (1.44,4.15) .. (2.30,4.24)
  .. controls (3.28,4.34) and (4.06,4.07) .. (5.00,4.20)
  .. controls (5.92,4.33) and (6.68,4.10) .. (7.62,4.23)
  .. controls (8.55,4.36) and (9.15,4.18) .. (10,4.16);

\draw[white,opacity=0.55,line width=0.52pt]
  (1.10,4.03) .. controls (1.72,4.18) and (2.18,3.96) .. (2.78,4.08);
\draw[white,opacity=0.50,line width=0.52pt]
  (3.28,4.16) .. controls (4.02,4.35) and (4.82,4.00) .. (5.48,4.15);
\draw[white,opacity=0.45,line width=0.52pt]
  (6.32,4.18) .. controls (6.92,4.02) and (7.58,4.26) .. (8.22,4.14);

\node[font=\Large\rmfamily,anchor=center] at (2.5,0.575) {anode};
\node[font=\Large\rmfamily,anchor=center] at (7.5,0.575) {cathode};
\node[font=\Huge\rmfamily,text=waterEdge!85!black] at (5.0,2.55) {$\Omega$};
\node[font=\Large\rmfamily] at (5.0,4.82) {air};

\node[mathlab,text=anodeDark!90!black] (Ga) at (2.10,1.86) {$\Gamma_a$};
\draw[anodeDark!90!black,pointer] (Ga.south) -- (2.65,1.19);
\fill[anodeDark!90!black] (2.65,1.19) circle[radius=0.06];

\node[mathlab,text=cathodeDark!95!black] (Gc) at (7.95,1.86) {$\Gamma_c$};
\draw[cathodeDark!95!black,pointer] (Gc.south) -- (7.35,1.19);
\fill[cathodeDark!95!black] (7.35,1.19) circle[radius=0.06];

\node[lab,text=black] (GN) at (8.85,5.26) {$\Gamma_N$};
\draw[black,line width=0.8pt,-{Latex[length=2.8mm,width=1.8mm]}] (GN.south west) -- (8.18,4.28);

\end{tikzpicture}
 }};

\coordinate (sliceStart) at ([xshift=-0.18cm,yshift=-1.08cm]threeD.north east);
\coordinate (sliceEnd) at ([xshift=-0.18cm]twoD.west);
\draw[black,line width=1.35pt,-{Latex[length=4.0mm,width=2.4mm]}]
    (sliceStart) .. controls +(1.00cm,0.00cm) and +(-1.10cm,0.06cm) .. (sliceEnd);
\end{tikzpicture}
 \caption{Three-dimensional corrosion sketch with an indicated slicing plane and the corresponding two-dimensional
cross-section used in the model. The reduced planar domain keeps the reactive anodic and cathodic boundary pieces
$\Gamma_a,\Gamma_c$ and the insulated boundary part $\Gamma_N$ that enter the mixed boundary-value problem.}
\label{fig:corrosion_illustration}
\end{paperfigure}

\subsection{Singular-limit results}\label{subsec:singular-limit-theorems}

In this subsection, we state the singular limit results in the geometric setting fixed by
\eqref{eq:boundary_decomp}--\eqref{eq:separation_condition} together with the boundary-value problem
\eqref{eq:intro_model_pde}. The first theorem is a convergence result: it identifies the limiting boundary trace and
the harmonic interior solution selected by that trace.

The second theorem is a sharp energy result: it shows that, at
smooth cathode-anode junctions, the leading energy is the half-plane boundary contribution of one jump, multiplied by the
number of junctions. These conclusions separate the global compactness mechanism from the local boundary-layer
calculation.

For each $\kappa>0$, the weak formulation of \eqref{eq:intro_model_pde} is to find
$\phi_\kappa\in H^1(\Omega)$ such that
\begin{equation}\label{eq:weak_formulation}
\int_\Omega \nabla \phi_\kappa\cdot \nabla v\,dx
+\frac{1}{\kappa}\int_{\Gamma_c} i_c(\phi_\kappa)\,v\,ds
+\frac{1}{\kappa}\int_{\Gamma_a} i_a(\phi_\kappa)\,v\,ds=0
\end{equation}
for every $v\in H^1(\Omega)$. The basic solvability and regularity properties are recorded in
Lemma~\ref{lem:basic_solvability_regularity}; in particular, $\phi_\kappa$ is well defined and unique for every fixed
$\kappa>0$.

We use $H^{1/2}(\Gamma_*)$ as the restriction space of $H^{1/2}(\partial\Omega)$, equipped with the quotient trace norm
\[
\|g\|_{H^{1/2}(\Gamma_*)}:=
\inf\{\|G\|_{H^{1/2}(\partial\Omega)}:\ G|_{\Gamma_*}=g\}.
\]
For $0\le s<1/2$, $H^s(\Gamma_*)$ denotes the usual fractional Sobolev space on the finite union of boundary arcs,
defined in arclength coordinates; endpoint values are immaterial in this range. The step trace $\Phi_0$ defined in
\eqref{eq:equilibrium_profile} belongs to $H^s(\Gamma_*)$ for every $0\le s<1/2$.

For $g\in H^{1/2}(\Gamma_*)$, $H_*g$ denotes the unique function $u\in H^1(\Omega)$ with
$\operatorname{Tr}u=g$ on $\Gamma_*$ such that
\[
\int_\Omega \nabla u\cdot\nabla v\,dx=0\qquad
\text{for all }v\in H^1(\Omega)\text{ with }\operatorname{Tr}v=0\text{ on }\Gamma_*.
\]
Thus $H_*g$ is the mixed harmonic extension with homogeneous Neumann data on $\Gamma_N$.

We now state the two main results of the paper.

\begin{theorem}[Boundary convergence in subcritical trace norms]\label{thm:boundary_harmonic_selection}
Assume the geometric setting \eqref{eq:boundary_decomp}--\eqref{eq:separation_condition}. Let Tr$\phi_\kappa$ be the solution trace of problem \eqref{eq:intro_model_pde}, and let $\Phi_0$ be the step function \eqref{eq:equilibrium_profile}. For every $0\le s<1/2$,
there is $C_s>0$ such that, for all sufficiently small
$\kappa>0$,
\begin{equation}\label{eq:hs_boundary_convergence}
\|\operatorname{Tr}\phi_\kappa-\Phi_0\|_{H^s(\Gamma_*)}
\le
C_s\,\kappa^{1/2-s}|\log\kappa|^{1/2},
\end{equation}

Thus $\operatorname{Tr}\phi_\kappa\to\Phi_0$ in $H^s(\Gamma_*)$ for every $0\le s<1/2$. The constants depend only on the domain,
the boundary decomposition, the equilibrium values $\phi_a,\phi_c$, the fixed boundary-reaction parameters, and on $s$.
\end{theorem}

The $s=0$ boundary estimate propagates to compact subsets through interior estimates for the corresponding mixed harmonic
extension.

\begin{corollary}[Interior convergence]\label{cor:interior_selection}
Under the hypotheses of Theorem~\ref{thm:boundary_harmonic_selection}, let
$u_0$ be the bounded mixed harmonic extension of the boundary values $\Phi_0$ on $\Gamma_*$, with homogeneous Neumann
data on $\Gamma_N$. Then, for every compact set $K\cc\Omega$ and every integer $m\ge0$, there is
$C_{K,m}>0$ such that, for all sufficiently small $\kappa>0$,
\[
\|\phi_\kappa-u_0\|_{C^m(K)}
\le C_{K,m}(\kappa|\log\kappa|)^{1/2}.
\]
Consequently, $\phi_\kappa\to u_0$ in $C^m_{\mathrm{loc}}(\Omega)$ for every integer $m\ge0$.
\end{corollary}

The energy associated with the same weak problem \eqref{eq:intro_model_pde} is
\begin{equation}\label{elec_vark}
J_\kappa(\phi):= \frac{1}{2}\int_\Omega |\nabla \phi|^2
+ \frac{1}{\kappa}\left(\int_{\Gamma_c}I_c(\phi(s))\,ds+\int_{\Gamma_a}I_a(\phi(s))\,ds\right),
\end{equation}
where $I_c(\phi):=\int_{\phi_c}^{\phi} i_c(r)\,dr$ and
$I_a(\phi):=\int_{\phi_a}^{\phi} i_a(r)\,dr$. These primitives are normalized to vanish at their respective
equilibria; they are nonnegative and strictly convex. Lemma~\ref{lem:basic_solvability_regularity} also identifies
$\phi_\kappa$ as the unique minimizer of $J_\kappa$.

Where Theorem~\ref{thm:boundary_harmonic_selection} and Corollary~\ref{cor:interior_selection} control the trace error
and the selected interior profile, the next theorem pins down the sharp logarithmic energy constant.

\begin{theorem}[Sharp logarithmic expansion]\label{thm:energy_scale}
Assume the geometric setting \eqref{eq:boundary_decomp}--\eqref{eq:separation_condition}. Let $J_\kappa$ be the
energy functional in \eqref{elec_vark}, and assume there are
$N=\#\Sigma_*$ junction points. Then
there is $C>0$ such that, for all sufficiently small $\kappa>0$,
\begin{equation}\label{eq:exact_log_constant}
\left|
J_\kappa(\phi_\kappa)
-\frac{N(\phi_c-\phi_a)^2}{2\pi}|\log\kappa|
\right|
\le C\log|\log\kappa|.
\end{equation}
Consequently
\[
\lim_{\kappa\to0^+}\frac{J_\kappa(\phi_\kappa)}{|\log\kappa|}
=\frac{N(\phi_c-\phi_a)^2}{2\pi}.
\]
The remainder constant depends only on the fixed domain geometry, the boundary decomposition, the equilibrium gap
$\phi_c-\phi_a$, and the fixed boundary-reaction parameters.
\end{theorem}

The preliminary lemmas are collected in Section~\ref{sec:preliminary-lemmas}, and the proofs of the singular-limit
theorems are given in Section~\ref{sec:proofs}.
 % !TEX root = ../main.tex
\section{Preliminary lemmas}\label{sec:preliminary-lemmas}

This section collects the analytic inputs used in the proofs: coercivity of the boundary primitives, range and
regularity properties of solutions, mixed harmonic extension estimates, and the energy bounds.

\begin{lemma}[Basic solvability and regularity]\label{lem:basic_solvability_regularity}
For every $\kappa>0$, equation \eqref{eq:intro_model_pde} admits a unique weak solution
$\phi\in H^1(\Omega)$. If $\Omega$ is $C^{1,1}$, then this solution belongs to $H^{3/2}(\Omega)$ and satisfies
\begin{equation}\label{reg_eq}
\|\phi\|_{H^{3/2}(\Omega)}\leq Be^{A\|\phi\|_{H^{1}(\Omega)}^2}+C'\|\phi\|_{L^2(\Omega)}
\end{equation}
for suitable constants $A,B,C'>0$, with $B$ allowed to depend on $\kappa$. If $\Omega$ is a Lipschitz
curvilinear polygon, then for every $\eta\in(0,1/2)$ the same estimate holds with
$H^{3/2-\eta}(\Omega)$ in place of $H^{3/2}(\Omega)$ and with constants allowed to depend on $\eta$.
\end{lemma}

The proof uses the following elementary coercivity estimate for the boundary primitives.

\begin{lemma}[Quadratic lower bounds for the primitives]\label{lem:quadratic_primitives}
The boundary primitives satisfy
\[
I_c(t)\ge \frac{m_c}{2}(t-\phi_c)^2,\qquad
I_a(t)\ge \frac{m_a}{2}(t-\phi_a)^2
\qquad\text{for all }t\in\mathbb R,
\]
where
\[
m_c:=i_{c,0}\min\{C_1,C_2\},\qquad
m_a:=i_{a,0}\min\{A_1,A_2\}.
\]
\end{lemma}

\begin{proof}[Proof of Lemma~\ref{lem:quadratic_primitives}]
Since
\[
I_c''(t)= i_{c,0}\big(C_1 e^{C_1(t-\phi_c)}+C_2 e^{-C_2(t-\phi_c)}\big)\ge m_c,
\]
and similarly $I_a''\ge m_a$, Taylor's theorem at the minimizing points gives the result:
$I_c(\phi_c)=I_c'(\phi_c)=0$ and $I_a(\phi_a)=I_a'(\phi_a)=0$.
\end{proof}

\begin{proof}[Proof of Lemma~\ref{lem:basic_solvability_regularity}]
We first record the exponential trace estimate needed below. The trace Moser--Trudinger inequality on the compact
one-dimensional Lipschitz manifold $\partial\Omega$ implies that, for every $\lambda>0$, there is $C_\lambda>0$
such that
\[
\int_{\partial\Omega} e^{\lambda |u|}\,ds
\le C_\lambda\exp\!\big(C_\lambda\|u\|_{H^{1/2}(\partial\Omega)}^2\big)
\qquad (u\in H^{1/2}(\partial\Omega)).
\]
Indeed, this follows from the normalized trace Moser--Trudinger inequality applied to $u-\bar u$, Young's
inequality, and the bound $|\bar u|\le C\|u\|_{H^{1/2}(\partial\Omega)}$; see
\cite{Cianchi2008,Yang2007SharpTraceTM}. Since the signed exponential nonlinearities and their primitives have at most
linear exponential growth, the trace theorem gives, for every $1\le p<\infty$ and every $\phi\in H^1(\Omega)$,
\[
\begin{aligned}
&\|I_c(\phi)\|_{L^p(\Gamma_c)}^p+\|I_a(\phi)\|_{L^p(\Gamma_a)}^p \\
&\quad+\|i_c(\phi)\|_{L^p(\Gamma_c)}^p+\|i_a(\phi)\|_{L^p(\Gamma_a)}^p
\le C_p\exp\!\big(C_p\|\phi\|_{H^1(\Omega)}^2\big).
\end{aligned}
\]
In particular, the boundary part of $J_\kappa$ is finite on $H^1(\Omega)$ and
the boundary nonlinearities belong to $L^p$ for every finite $p$.

We now prove existence and uniqueness. Lemma~\ref{lem:quadratic_primitives} gives constants $m_c,m_a>0$ such that
\[
J_\kappa(\phi)\ge \frac12\|\nabla\phi\|_{L^2(\Omega)}^2
+\frac{c}{\kappa}\|\operatorname{Tr}\phi-\Phi_0\|_{L^2(\Gamma_a\cup\Gamma_c)}^2,
\qquad c:=\frac12\min\{m_a,m_c\}>0.
\]
Since $\Gamma_a\cup\Gamma_c$ has positive boundary measure, the Poincar\'e trace inequality
\[
\|\phi\|_{L^2(\Omega)}^2
\le C\bigl(\|\nabla\phi\|_{L^2(\Omega)}^2+\|\operatorname{Tr}\phi\|_{L^2(\Gamma_a\cup\Gamma_c)}^2\bigr)
\]
and the boundedness of $\Phi_0$ imply coercivity of $J_\kappa$ on $H^1(\Omega)$.

Let $(\phi_n)$ be a minimizing sequence. Coercivity gives $\phi_n\rightharpoonup\phi$ in $H^1(\Omega)$ after passing
to a subsequence, and the compact trace embedding gives $\operatorname{Tr}\phi_n\to\operatorname{Tr}\phi$ strongly in
$L^2(\partial\Omega)$ and a.e. on the boundary. The boundedness of $(\phi_n)$ in $H^1$ and
the exponential trace estimate above, applied with some $q>1$, imply an $L^q$ bound on
$I_c(\phi_n)$ and $I_a(\phi_n)$ on the corresponding boundary arcs. These boundary primitives are therefore uniformly
integrable. Thus, since the traces converge a.e.,
\[
\int_{\Gamma_c}I_c(\phi_n)\,ds+\int_{\Gamma_a}I_a(\phi_n)\,ds
\to
\int_{\Gamma_c}I_c(\phi)\,ds+\int_{\Gamma_a}I_a(\phi)\,ds,
\]
while the Dirichlet term is weakly lower semicontinuous. Hence $J_\kappa$ attains its minimum.

The functional is strictly convex. Indeed, equality in the convexity inequality for the Dirichlet term forces two
competitors to differ by a constant, and equality in the boundary terms is then impossible unless that constant is
zero, because $I_c''>0$ and $I_a''>0$ on the positive-measure sets $\Gamma_c$ and $\Gamma_a$. Hence the minimizer is unique. Its first
variation is well defined since the exponential trace estimate gives $i_c(\phi),i_a(\phi)\in L^2$ on the boundary and the
trace of every $v\in H^1(\Omega)$ lies in $L^2(\partial\Omega)$. The minimizer therefore satisfies
\eqref{eq:weak_formulation}, which is the weak form of \eqref{eq:intro_model_pde}.

It remains to prove the stated regularity. Let $\phi$ be the unique weak solution. The boundary datum
\[
F_\phi=
\begin{cases}
0 & \text{on }\Gamma_N,\\
-\dfrac{i_c(\phi)}{\kappa} & \text{on }\Gamma_c,\\
-\dfrac{i_a(\phi)}{\kappa} & \text{on }\Gamma_a
\end{cases}
\]
belongs to $L^2(\partial\Omega)$ by the exponential trace estimate above. More explicitly,
\[
\|F_\phi\|_{L^2(\partial\Omega)}
\le \frac{\|i_c(\phi)\|_{L^2(\Gamma_c)}+\|i_a(\phi)\|_{L^2(\Gamma_a)}}{\kappa}
\le B e^{A\|\phi\|_{H^1(\Omega)}^2},
\]
with $B$ allowed to depend on the fixed value of $\kappa$. Taking $v=1$ in \eqref{eq:weak_formulation} gives the compatibility
condition $\int_{\partial\Omega}F_\phi\,ds=0$. Standard Neumann regularity for
$C^{1,1}$ domains \cite{LionsMagenes1972,Grisvard2011} gives
\[
\|\phi\|_{H^{3/2}(\Omega)}\le C' \bigl(\|F_\phi\|_{L^2(\partial \Omega)}+\|\phi\|_{L^2(\Omega)}\bigr),
\]
which proves the estimate in the smooth-boundary case. For Lipschitz curvilinear polygons we use the corresponding
Neumann shift theorem below the corner threshold: for every $\eta>0$,
\[
\|\phi\|_{H^{3/2-\eta}(\Omega)}
\le C'_\eta \bigl(\|F_\phi\|_{L^2(\partial \Omega)}+\|\phi\|_{L^2(\Omega)}\bigr),
\]
again from \cite{Grisvard2011}. Since every Lipschitz polygonal angle lies in $(0,2\pi)$, the first Neumann corner
exponents remain larger than $1/2$, which is enough for the order used here.
This proves the curvilinear polygonal statement.
\end{proof}

The remaining properties of the PDE solution are grouped in the following lemma.

\begin{lemma}[Bounds on the solution and optimal regularity]\label{lem:variational_tools}
Assume that $\Omega$ and the boundary decomposition satisfy
\eqref{eq:boundary_decomp}--\eqref{eq:separation_condition}, and let \(\phi\) be the weak solution of
\eqref{eq:intro_model_pde}, equivalently \eqref{eq:weak_formulation}.
\begin{enumerate}
\item If, in addition, $\Omega$ is connected, then
\[
\phi_a\le \phi(x)\le \phi_c
\qquad\text{for a.e. }x\in\Omega.
\]

\item In the $C^{1,1}$ case, the solution satisfies $\phi\in H^s(\Omega)$ for all $0\le s<2$. In the Lipschitz
curvilinear polygonal case, $\phi\in H^{3/2-\eta}(\Omega)$ for every $\eta\in(0,1/2)$, and
$\phi\in H^s(U\cap\Omega)$ for all $0\le s<2$ whenever $\overline U$ is disjoint from the corner set
$\mathcal C_\Omega$. Moreover, if the induced Neumann datum has a jump at a smooth boundary point, then
$\phi\notin H^2(\Omega)$.
\end{enumerate}
\end{lemma}

We address the two solution-dependent items in order. The proof uses only the monotonicity structure of the signed
exponential nonlinearities and the standard boundary regularity machinery already invoked in
Lemma~\ref{lem:basic_solvability_regularity}.

\begin{proof}[Proof of Lemma~\ref{lem:variational_tools}]
\setcounter{proofstep}{0}

\begin{step}{Range bound}
For the range bound, test \eqref{eq:weak_formulation} with $v=(\phi-\phi_c)^+$. On $\{v>0\}$, both $i_c(\phi)$ and
$i_a(\phi)$ are nonnegative because $\phi>\phi_c>\phi_a$. Hence all terms in the weak identity are nonnegative, so
\[
0=\int_{\{v>0\}}|\nabla\phi|^2\,dx
+\frac1\kappa\int_{\Gamma_c}i_c(\phi)v\,ds
+\frac1\kappa\int_{\Gamma_a}i_a(\phi)v\,ds
\]
forces each term to vanish. Thus $v$ is constant on the connected domain $\Omega$. If this constant were positive,
then the trace of $\phi$ would be strictly larger than both limiting values on the positive-measure set
$\Gamma_c\cup\Gamma_a$, making the boundary integrals strictly positive. Hence $v=0$ and $\phi\le\phi_c$. Testing
with $v=-(\phi_a-\phi)^+$ gives the lower bound in the same way: on $\{\phi<\phi_a\}$ the functions $i_a(\phi)$,
$i_c(\phi)$, and $v$ are all
nonpositive, so the products $i_c(\phi)v$ and $i_a(\phi)v$ are nonnegative.
\end{step}

\begin{step}{Regularity below $H^2$}
The case $s=0$ is immediate from $\phi\in H^1(\Omega)$. For the positive orders below $2$, choose
$\eta\in(0,1/2)$. Lemma~\ref{lem:basic_solvability_regularity} gives a trace in
$H^{1-\eta}$ on each relatively open boundary arc; in the $C^{1,1}$ case one may take $\eta=0$ at this first step.

On such an arc the range bound keeps $\phi$ in the compact interval $[\phi_a,\phi_c]$, and the derivatives of the
signed exponential nonlinearities are bounded on this interval. The Sobolev chain rule therefore gives
$i_a(\phi),i_c(\phi)\in H^{1-\eta}$ along their respective arcs.

Thus the induced Neumann datum $F_\phi$ is piecewise $H^{1-\eta}$, with only finitely many possible jumps at arc
endpoints. A one-dimensional piecewise $H^{1-\eta}$ function with finitely many jumps belongs to
$H^s(\partial\Omega)$ for every $s<1/2$: this follows by writing it as a finite sum of $H^{1-\eta}$ pieces cut off by
interval indicator functions, and using $\mathbf{1}_I\in H^s$ exactly for $s<1/2$.

The standard Neumann shift theorem on $C^{1,1}$ domains yields $\phi\in H^{s+3/2}(\Omega)$ for every $s<1/2$, hence
$\phi\in H^r(\Omega)$ for all $r<2$ in the smooth-boundary case. In the Lipschitz curvilinear polygonal case,
Lemma~\ref{lem:basic_solvability_regularity} gives the stated global $H^{3/2-\eta}$ regularity. Away from the finite
corner set, the boundary is $C^{1,1}$ and the same local shift theorem gives
$\phi\in H^r(U\cap\Omega)$ for all $r<2$ whenever $\overline U\cap\mathcal C_\Omega=\emptyset$. Reentrant corners may
limit the global Sobolev order below $2$, so the polygonal conclusion is local near smooth boundary points.
\end{step}

\begin{step}{Failure of $H^2$}
It remains to identify when a jump is actually present. Taking $\eta<1/2$ above, the trace has a continuous
representative in a one-dimensional boundary coordinate across every smooth junction. Let $x_j\in\Sigma_*$ be a
smooth $\Gamma_c$--$\Gamma_a$ junction and write
$\tau_j:=\operatorname{Tr}\phi(x_j)$. The one-sided limits of the induced Neumann datum at $x_j$ are
\[
-i_c(\tau_j)/\kappa
\qquad\text{and}\qquad
-i_a(\tau_j)/\kappa,
\]
in the order determined by the adjacent $\Gamma_c$ and $\Gamma_a$ arcs. The range bound gives
$\tau_j\in[\phi_a,\phi_c]$. On this interval,
\[
\begin{cases}
i_c(t)\le0,\quad i_c(t)=0\Longleftrightarrow t=\phi_c,\\
i_a(t)\ge0,\quad i_a(t)=0\Longleftrightarrow t=\phi_a,
\end{cases}
\]
because the signed exponential nonlinearities are strictly increasing and vanish only at $\phi_c$ and $\phi_a$,
respectively. Since $\phi_a<\phi_c$, the equalities $i_c(\tau_j)=i_a(\tau_j)$ would force
$\tau_j=\phi_c=\phi_a$, which is impossible. Hence the two one-sided flux limits differ, and $F_\phi$ has a jump at
every smooth point where $\Gamma_a$ and $\Gamma_c$ meet.

At a smooth endpoint where one of $\Gamma_a,\Gamma_c$ meets $\Gamma_N$, the same argument shows the exact criterion: the one-sided limits are $0$ and
$-i_a(\tau)/\kappa$ at a $\Gamma_a$-Neumann endpoint, or $0$ and $-i_c(\tau)/\kappa$ at a $\Gamma_c$-Neumann
endpoint. Thus such an endpoint is a jump of $F_\phi$ precisely when $\tau\ne\phi_a$ in the $\Gamma_a$ case, respectively
$\tau\ne\phi_c$ in the $\Gamma_c$ case.

Finally, suppose one of these jumps occurs at a smooth boundary point and $\phi\in H^2(\Omega)$. In a local boundary
coordinate the normal is continuous, so the trace theorem gives
$\partial_\nu\phi\in H^{1/2}$ on that boundary neighborhood. This contradicts the fact that a one-dimensional jump is
not in $H^{1/2}$. Hence $\phi\notin H^2(\Omega)$ whenever $F_\phi$ has a jump at a smooth $\Gamma_c$--$\Gamma_a$
junction.
\end{step}
\end{proof}

\begin{lemma}[Mixed extension and Poisson smoothing]\label{lem:star_extension_poisson}
Let $H_*$ be the mixed harmonic extension from Subsection~\ref{subsec:singular-limit-theorems}. Then, for every
$g\in H^{1/2}(\Gamma_*)$,
\begin{equation}\label{eq:star_extension_bound}
\|\nabla H_*g\|_{L^2(\Omega)}\le C\|g\|_{H^{1/2}(\Gamma_*)}.
\end{equation}
The same mixed problem has a Poisson extension $\mathcal P_*$, agreeing with $H_*$
on $H^{1/2}(\Gamma_*)$ and with the harmonic-measure extension on bounded Borel data, such that
\begin{equation}\label{eq:poisson_l2_smoothing}
\|\mathcal P_*g\|_{C^m(K)}\le C_{K,m}\|g\|_{L^2(\Gamma_*)}
\end{equation}
for every $K\cc\Omega$, $m\ge0$, and $g\in L^2(\Gamma_*)$.
\end{lemma}

\begin{proof}[Proof of Lemma~\ref{lem:star_extension_poisson}]
For the $H^{1/2}$ extension estimate, choose an $H^{1/2}(\partial\Omega)$ representative $G$ of $g$ and an
$H^1(\Omega)$ lift $\widetilde G$. The subspace of functions in $H^1(\Omega)$ whose trace vanishes on $\Gamma_*$ is
closed, and the trace Poincar\'e inequality on the positive-measure set $\Gamma_*$ gives
\[
\|v\|_{H^1(\Omega)}\le C\|\nabla v\|_{L^2(\Omega)}
\qquad (v\in H^1(\Omega),\ \operatorname{Tr}v=0\text{ on }\Gamma_*).
\]
Thus the bilinear form $\int_\Omega\nabla z\cdot\nabla v$ is coercive on this closed subspace. Lax--Milgram gives a
unique $z\in H^1(\Omega)$ with $\operatorname{Tr}z=0$ on $\Gamma_*$ such that
\[
\int_\Omega \nabla z\cdot\nabla v\,dx=-\int_\Omega\nabla\widetilde G\cdot\nabla v\,dx
\qquad(v\in H^1(\Omega),\ \operatorname{Tr}v=0\text{ on }\Gamma_*),
\]
and $H_*g:=\widetilde G+z$ is independent of the chosen lift. The same coercivity estimate gives
\[
\|\nabla H_*g\|_{L^2(\Omega)}\le C\|\widetilde G\|_{H^1(\Omega)}
\le C\|G\|_{H^{1/2}(\partial\Omega)}.
\]
Taking the infimum over all representatives $G$ proves \eqref{eq:star_extension_bound}.

The mixed problem under the finite decomposition used here admits harmonic-measure, equivalently Poisson-kernel,
representations for boundary data on the Dirichlet part; see
\cite{Kenig1994,Brown1994,OttBrown2013,TaylorOttBrown2013}. The kernel is associated with homogeneous Neumann data on
$\Gamma_N$ and Dirichlet data on $\Gamma_*$. For fixed $K\cc\Omega$, interior estimates applied to the kernel in the
$x$ variable imply that the kernel and all of its $x$-derivatives are square-integrable in the boundary variable,
uniformly for $x\in K$.

Hence, the Poisson extension is well defined for $L^2(\Gamma_*)$ data and, for $|\alpha|\le m$,
\[
\partial_x^\alpha \mathcal P_*g(x)=\int_{\Gamma_*}\partial_x^\alpha P_*(x,y)g(y)\,ds_y,
\qquad
\sup_{x\in K}\|\partial_x^\alpha P_*(x,\cdot)\|_{L^2(\Gamma_*)}\le C_{K,m},
\]
which proves \eqref{eq:poisson_l2_smoothing} by Cauchy's inequality.
When $g\in H^{1/2}(\Gamma_*)$, uniqueness of the finite-energy mixed solution and the same boundary trace show that
this Poisson extension agrees with the variational extension constructed above. For bounded Borel data the
harmonic-measure formula is unchanged by altering values at finitely many endpoints.
\end{proof}

\begin{lemma}[Local annular lower bound near a smooth boundary jump]\label{lem:local_annular_lower_bound}
Let $\Psi:B_R^+\to\Omega\cap U$ be a $C^{1,1}$ boundary-flattening chart with $\Psi(0)\in\partial\Omega$ and
$D\Psi(0)$ orthogonal, whose positive and negative flat boundary rays map, after a rigid motion, to two adjacent
components of $\Gamma_*$. Fix $q_\pm\in\mathbb R$ and $M>0$. Then, every
$u\in H^1(\Omega)$ with $|u|\le M$ a.e. in $\Psi(B_R^+)$ satisfies, for $0<r<R/2$,
\begin{equation}\label{eq:local_annular_lower_bound}
\frac12\int_{\Psi(A_{r,R}^+)}|\nabla u|^2\,dx
\ge
\frac{|q_+-q_-|^2}{2\pi}\log\frac{R}{r}
-C\left(1+E r^{-1/2}\right),
\end{equation}
where $A_{r,R}^+:=\{(\rho,\theta):r<\rho<R,\ 0<\theta<\pi\}$ and, for $w=u\circ\Psi$,
\[
\begin{gathered}
b(\rho):=\operatorname{Tr}w(\rho,0),\qquad
a(\rho):=\operatorname{Tr}w(-\rho,0)\quad\text{for a.e. }0<\rho<R,\\
E^2:=\int_0^R\left(|a(\rho)-q_-|^2+|b(\rho)-q_+|^2\right)\,d\rho .
\end{gathered}
\]
\end{lemma}

\begin{proof}[Proof of Lemma~\ref{lem:local_annular_lower_bound}]
The trace and slicing statements used below are standard consequences of applying the Sobolev trace theorem on
Lipschitz sectors and then using Fubini in polar coordinates. In particular, for a.e. $\rho\in(0,R)$ the angular slice
$\theta\mapsto w(\rho,\theta)$ belongs to $H^1(0,\pi)$, its endpoint values agree with the flat boundary traces of
$w$, and
\[
\int_r^R\int_0^\pi \rho^{-1}|\partial_\theta w(\rho,\theta)|^2\,d\theta\,d\rho<\infty .
\]

For $y=(s,t)=\rho(\cos\theta,\sin\theta)$, the change of variables gives
\[
\int_{\Psi(A_{r,R}^+)}|\nabla u|^2\,dx
=\int_{A_{r,R}^+}A(y)\nabla w(y)\cdot\nabla w(y)\,dy,
\]
where
\[
A(y)=|\det D\Psi(y)|\,D\Psi(y)^{-1}D\Psi(y)^{-T}.
\]
After composing $\Psi$ with a rigid motion, $A(0)=I$. The $C^{1,1}$ bounds give
$A(y)\xi\cdot\xi\ge(1-C\rho)|\xi|^2$. Reducing $R$ if necessary and keeping only the angular derivative,
\[
\int_{\Psi(A_{r,R}^+)}|\nabla u|^2\,dx
\ge
\int_r^R\frac{1-C\rho}{\rho}
\int_0^\pi |\partial_\theta w(\rho,\theta)|^2\,d\theta\,d\rho .
\]
The one-dimensional endpoint inequality on each angular slice gives
\[
\int_0^\pi |\partial_\theta w(\rho,\theta)|^2\,d\theta
\ge \frac{|b(\rho)-a(\rho)|^2}{\pi}.
\]
Therefore
\[
\int_{\Psi(A_{r,R}^+)}|\nabla u|^2\,dx
\ge
\int_r^R \frac{|b(\rho)-a(\rho)|^2}{\pi\rho}\,d\rho
-C\int_r^R |b(\rho)-a(\rho)|^2\,d\rho .
\]
The last integral is bounded by a constant depending only on $M$ and $R$. Let $d:=|q_+-q_-|$. Since
\[
b-a=(q_+-q_-)+(b-q_+)-(a-q_-),
\]
we have
\[
|b-a|^2\ge d^2-2d\left(|b-q_+|+|a-q_-|\right).
\]
By Cauchy's inequality,
\[
\int_r^R\frac{|b-q_+|+|a-q_-|}{\rho}\,d\rho
\le
C E\left(\int_r^R\rho^{-2}\,d\rho\right)^{1/2}
\le C E r^{-1/2}.
\]
Combining these estimates and dividing by two proves \eqref{eq:local_annular_lower_bound}.
\end{proof}

\section{Proofs}\label{sec:proofs}

We devote this section to the proofs of the main results. We begin with the proof of the singular-limit statement in
Theorem~\ref{thm:boundary_harmonic_selection}.

\subsection{Proof of Theorem~\ref{thm:boundary_harmonic_selection}}\label{subsec:proof-boundary-selection}

\begin{proof}[Proof of Theorem~\ref{thm:boundary_harmonic_selection}]
We divide the proof in three steps.
\setcounter{proofstep}{0}

\begin{step}{Mollified boundary traces and boundary control}
We first construct the only competitor needed for this proof. Under
\eqref{eq:boundary_decomp}--\eqref{eq:separation_condition}, there are
$\varepsilon_0\in(0,1)$ and $c,C>0$ such that for each
$0<\varepsilon<\varepsilon_0$ one can choose
$g^\varepsilon\in H^{1/2}(\Gamma_*)$ satisfying
\begin{equation}\label{eq:mollified_step_l2_hhalf}
\|g^\varepsilon-\Phi_0\|_{L^2(\Gamma_*)}^2\le C\varepsilon,
\qquad
\|g^\varepsilon\|_{H^{1/2}(\Gamma_*)}^2\le C|\log\varepsilon|,
\end{equation}
and, for every $0\le s<1/2$,
\begin{equation}\label{eq:mollified_step_hs}
\|g^\varepsilon-\Phi_0\|_{H^s(\Gamma_*)}\le C_s\varepsilon^{1/2-s}.
\end{equation}
Moreover, $g^\varepsilon=\Phi_0$ outside an $O(\varepsilon)$ neighborhood of
$\Sigma_*$ and $\phi_a\le g^\varepsilon\le\phi_c$. If
$\phi^\varepsilon:=H_*g^\varepsilon$, then
\begin{equation}\label{eq:mollified_step_energy}
\|\nabla\phi^\varepsilon\|_{L^2(\Omega)}^2\le C|\log\varepsilon|.
\end{equation}
Finally, every $\phi\in H^1(\Omega)$ satisfies the boundary-control estimate
\begin{equation}\label{eq:boundary_control_inline}
\kappa J_\kappa(\phi)\ge c\|\operatorname{Tr}\phi-\Phi_0\|_{L^2(\Gamma_*)}^2 .
\end{equation}

When junctions are present, the construction is by one-dimensional mollification of the boundary step around each
point of $\Sigma_*$, with the constant values preserved away from the mollification windows; see
Figure~\ref{fig:mollified_step}.

\begin{paperfigure}[t]
\centering
\begin{tikzpicture}[x=1.0cm,y=1.0cm,>=Latex]

\draw[->,thick] (-3.2,0) -- (3.35,0) node[right] {$s$};
\draw[->,thick] (-3.0,-0.25) -- (-3.0,2.65);

\draw[dashed,red,very thick] (-3,0.65) -- (0,0.65);
\draw[dashed,red,very thick] (0,0.65) -- (0,2.0);
\draw[dashed,red,very thick] (0,2.0) -- (3,2.0);

\draw[black,very thick]
  (-3,0.65) -- (-1,0.65)
  .. controls (-0.62,0.65) and (-0.55,0.82) .. (-0.34,1.02)
  .. controls (-0.10,1.25) and (0.10,1.40) .. (0.34,1.62)
  .. controls (0.55,1.82) and (0.62,2.0) .. (1,2.0)
  -- (3,2.0);

\draw[dotted,blue,thick] (-1,0) -- (-1,2.28);
\draw[dotted,blue,thick] (1,0) -- (1,2.28);
\node[below] at (-1,0) {$-\varepsilon$};
\node[below] at (1,0) {$\varepsilon$};
\node[below] at (0,0) {$0$};

\node[left] at (-3,0.65) {$\phi_a$};
\node[left] at (-3,2.0) {$\phi_c$};
\node[red] at (2.35,2.28) {$\Phi_0$};
\node[black] at (1.6,1.25) {$g^\varepsilon$};
\draw[<->,blue,thick] (-1,-0.68) -- (1,-0.68)
  node[midway,below,blue] {transition window};
\end{tikzpicture}
 \caption{Local smoothing of the limiting boundary step near one point of $\Sigma_*$. The trace is changed only on the window $|s|\le\varepsilon$, while it remains equal to the limiting values $\phi_a$ and $\phi_c$ outside this window.}
\label{fig:mollified_step}
\end{paperfigure}

If $\Sigma_*=\emptyset$, set $g^\varepsilon=\Phi_0$. Since $\Phi_0$ is then constant up to the endpoints of
each component of $\Gamma_*$, a fixed piecewise $H^1$ extension across the $\Gamma_N$ segments gives a
representative on $\partial\Omega$ with $H^{1/2}$ norm bounded independently of $\varepsilon$; the $L^2$ and
subcritical errors vanish.

Assume now that $\Sigma_*\ne\emptyset$. By the separation condition, choose pairwise disjoint arclength intervals
$(-2L_j,2L_j)$ around the jumps, contained in $\Gamma_*$ and disjoint from $\overline{\Gamma_N}$, with $L_j$
independent of $\varepsilon$. Then fix $\varepsilon_0$ so small that $2\varepsilon<L_j$ in every selected interval
whenever $0<\varepsilon<\varepsilon_0$. In the interval centered at a jump, write the local coordinate as $s=0$ and
replace the step by
\[
f^\varepsilon(s)=\phi_a+(\phi_c-\phi_a)\eta(s/\varepsilon),
\]
after interchanging $\phi_a,\phi_c$ if the orientation is reversed. Here $\eta\in C^\infty(\mathbb R)$ is nondecreasing,
$0\le\eta\le1$, $\eta=0$ on $(-\infty,-1]$, and $\eta=1$ on $[1,\infty)$. Away from the selected intervals we keep
the exact constant values on $\Gamma_*$. Thus
$g^\varepsilon$ changes $\Phi_0$ only in the union of transition intervals of total length $O(\varepsilon)$ and remains in
$[\phi_a,\phi_c]$. Since $|g^\varepsilon-\Phi_0|\le|\phi_c-\phi_a|$ pointwise, this gives the $L^2$ part of
\eqref{eq:mollified_step_l2_hhalf}.

The same local scaling gives the subcritical fractional approximation. For one transition, the difference
$f^\varepsilon-\Phi_0$ has the form $F(s/\varepsilon)$ on $(-\varepsilon,\varepsilon)$ and is zero outside, where
$F$ is a fixed bounded compactly supported step-profile difference. Since $F\in H^s(\mathbb R)$ for every
$s<1/2$,
\[
\|f^\varepsilon-\Phi_0\|_{H^s(-L,L)}\le C_s\varepsilon^{1/2-s},
\qquad 0\le s<1/2.
\]
The transition intervals are disjoint and finite in number, while cross terms between different intervals are
separated by a fixed distance. Summing over the transitions proves \eqref{eq:mollified_step_hs}.

It remains to record the $H^{1/2}$ cost. For one smoothed jump on a fixed interval $(-L,L)$, it is enough to estimate
the Gagliardo seminorm
\[
\iint_{(-L,L)^2}\frac{|f^\varepsilon(s)-f^\varepsilon(t)|^2}{|s-t|^2}\,ds\,dt
\le C|\log\varepsilon|.
\]
Split the integral into the regions $|s-t|\le\varepsilon$ and $|s-t|>\varepsilon$. On the first region,
$f^\varepsilon$ is $C/\varepsilon$-Lipschitz and is nonconstant only when at least one point lies in
$(-\varepsilon,\varepsilon)$. Hence
\[
\iint_{\substack{|s-t|\le\varepsilon\\ (-\varepsilon,\varepsilon)\cap\{s,t\}\ne\emptyset}}
\frac{|f^\varepsilon(s)-f^\varepsilon(t)|^2}{|s-t|^2}\,ds\,dt
\le
C\int_{-\varepsilon}^{\varepsilon}\int_{|h|\le\varepsilon}\frac{1}{\varepsilon^2}\,dh\,ds
\le C .
\]
On the second region, $|f^\varepsilon(s)-f^\varepsilon(t)|\le C$ and the only logarithmic contribution comes from
pairs lying on opposite sides of the transition, with separation at least $\varepsilon$. This gives
\[
C\int_{\varepsilon}^{2L}\frac{dr}{r}\le C|\log\varepsilon|.
\]
All other pairs either see the same constant value or are separated from the transition by a fixed distance, hence
contribute $O(1)$. Pairs belonging to two different transition-point neighborhoods are also separated by a fixed
distance and contribute only $O(1)$ because $g^\varepsilon$ is uniformly bounded.

To estimate the quotient norm, we construct one admissible representative on the full boundary. For each selected
transition interval choose a slightly larger arclength interval compactly contained in the same component of
$\Gamma_*$, and let $\chi_j$ be a fixed smooth cutoff which is equal to one on the transition interval and vanishes
outside the larger one. On $\Gamma_*$ we keep the values of $g^\varepsilon$. On the rest of $\partial\Omega$ we define
a function $G^\varepsilon$ as follows: on each $\Gamma_N$ segment, use a fixed piecewise $H^1$ interpolant joining the
two endpoint values inherited from the adjacent components of $\Gamma_*$. Thus
$G^\varepsilon|_{\Gamma_*}=g^\varepsilon$, so by the definition of the quotient norm,
\[
\|g^\varepsilon\|_{H^{1/2}(\Gamma_*)}
\le \|G^\varepsilon\|_{H^{1/2}(\partial\Omega)}.
\]

It remains only to bound the right-hand side. In arclength coordinates, multiplication by each fixed cutoff
$\chi_j$ is bounded on $H^{1/2}$, with a constant independent of $\varepsilon$. Hence the pieces containing the
smoothed transitions have the same $C|\log\varepsilon|$ squared $H^{1/2}$ bound obtained above. The remaining pieces
of $\Gamma_*$ are constant, while the interpolants on the $\Gamma_N$ segments are fixed piecewise $H^1$ functions; all
of these contribute $O(1)$ to the $H^{1/2}(\partial\Omega)$ norm. Finally, cross terms between distinct pieces are
uniformly bounded because their supports are either separated by a fixed positive arclength distance or the traces
match at the common endpoint through the chosen fixed interpolant. Therefore
\[
\|G^\varepsilon\|_{H^{1/2}(\partial\Omega)}^2\le C|\log\varepsilon|,
\]
and the quotient norm has the same bound. This proves the $H^{1/2}$ part of
\eqref{eq:mollified_step_l2_hhalf}.

The energy estimate \eqref{eq:mollified_step_energy} follows from
Lemma~\ref{lem:star_extension_poisson}, specifically \eqref{eq:star_extension_bound}, and
\eqref{eq:mollified_step_l2_hhalf}. To prove \eqref{eq:boundary_control_inline}, let $m_c,m_a>0$ be the lower
curvature constants supplied by Lemma~\ref{lem:quadratic_primitives}. Then
\[
\int_{\Gamma_c} I_c(\phi)\,ds+\int_{\Gamma_a} I_a(\phi)\,ds
\ge \frac12\min\{m_c,m_a\}\|\operatorname{Tr}\phi-\Phi_0\|_{L^2(\Gamma_*)}^2.
\]
Since $\Gamma_*=\Gamma_c\cup\Gamma_a$ up to finitely many endpoints and the Dirichlet term in $J_\kappa$ is
nonnegative, \eqref{eq:boundary_control_inline} follows after renaming the positive constant.
\end{step}

\begin{step}{The $L^2$ trace estimate}
Using the competitor $\phi^\varepsilon=H_*g^\varepsilon$ from Step 1,
\[
\kappa J_\kappa(\phi^\varepsilon)
\le C\kappa|\log\varepsilon|
+\int_{\Gamma_c}I_c(g^\varepsilon)\,ds+\int_{\Gamma_a}I_a(g^\varepsilon)\,ds.
\]
The boundary integrals vanish where $g^\varepsilon=\Phi_0$ because $I_c(\phi_c)=I_a(\phi_a)=0$. On the remaining arcs
the total length is $O(\varepsilon)$, and $I_a,I_c$ are bounded on $[\phi_a,\phi_c]$. Thus
$\kappa J_\kappa(\phi^\varepsilon)\le C(\kappa|\log\varepsilon|+\varepsilon)$. Minimality of $\phi_\kappa$ and the
boundary-control estimate \eqref{eq:boundary_control_inline} yield
\[
\|\operatorname{Tr}\phi_\kappa-\Phi_0\|^2_{L^2(\Gamma_*)}
\le C\bigl(\kappa|\log\varepsilon|+\varepsilon\bigr).
\]
Taking $\varepsilon=\kappa$ gives the $s=0$ estimate
\begin{equation}\label{convrategood}
\|\operatorname{Tr}\phi_\kappa-\Phi_0\|^2_{L^2(\Gamma_*)}\le C\kappa|\log\kappa|.
\end{equation}
The same comparison also gives
\[
J_\kappa(\phi_\kappa)\le C|\log\kappa|.
\]
The Dirichlet part of this energy bound, the trace theorem, and the range bound from
Lemma~\ref{lem:variational_tools}(i) imply
\[
\|\operatorname{Tr}\phi_\kappa\|_{L^2(\Gamma_*)}\le C,\qquad
\|\operatorname{Tr}\phi_\kappa\|_{H^{1/2}(\partial\Omega)}\le C|\log\kappa|^{1/2}.
\]
\end{step}

\begin{step}{Upgrade to subcritical trace norms}
Fix $0<s<1/2$, and let $g^\kappa$ be the mollified trace from Step 1 with $\varepsilon=\kappa$. Write
\[
\operatorname{Tr}\phi_\kappa-\Phi_0
=\bigl(\operatorname{Tr}\phi_\kappa-g^\kappa\bigr)+\bigl(g^\kappa-\Phi_0\bigr).
\]
\eqref{eq:mollified_step_l2_hhalf}--\eqref{eq:mollified_step_hs} give
\[
\|g^\kappa-\Phi_0\|_{H^s(\Gamma_*)}\le C_s\kappa^{1/2-s},
\qquad
\|g^\kappa\|_{H^{1/2}(\Gamma_*)}\le C|\log\kappa|^{1/2}.
\]
Moreover, by \eqref{convrategood} and the $L^2$ part of \eqref{eq:mollified_step_l2_hhalf},
\[
\|\operatorname{Tr}\phi_\kappa-g^\kappa\|_{L^2(\Gamma_*)}
\le C(\kappa|\log\kappa|)^{1/2},
\]
and the preceding trace bound gives
\[
\|\operatorname{Tr}\phi_\kappa-g^\kappa\|_{H^{1/2}(\Gamma_*)}
\le C|\log\kappa|^{1/2}.
\]
By real interpolation on quotient spaces \cite{LionsMagenes1972,AdamsFournier2003} and the endpoint extension result
for the subcritical range \cite{hitch}, this last difference satisfies
\[
\|\operatorname{Tr}\phi_\kappa-g^\kappa\|_{H^s(\Gamma_*)}
\le
C_s
\|\operatorname{Tr}\phi_\kappa-g^\kappa\|_{L^2(\Gamma_*)}^{1-2s}
\|\operatorname{Tr}\phi_\kappa-g^\kappa\|_{H^{1/2}(\Gamma_*)}^{2s}
\le C_s\kappa^{1/2-s}|\log\kappa|^{1/2}.
\]
Combining the last two estimates gives \eqref{eq:hs_boundary_convergence}. The case $s=0$ is exactly
\eqref{convrategood}. This proves convergence in $H^s(\Gamma_*)$ for every $0\le s<1/2$.
\end{step}
\end{proof}

\subsection{Proof of Corollary~\ref{cor:interior_selection}}\label{subsec:proof-interior-selection}

\begin{proof}[Proof of Corollary~\ref{cor:interior_selection}]
Let $g_\kappa:=\operatorname{Tr}\phi_\kappa-\Phi_0$ on $\Gamma_*$. Since
$u_0=\mathcal P_*\Phi_0$ and
$\phi_\kappa=\mathcal P_*(\operatorname{Tr}\phi_\kappa)$, linearity of the Poisson extension from
Lemma~\ref{lem:star_extension_poisson} gives
\[
\phi_\kappa-u_0=\mathcal P_*g_\kappa .
\]
Indeed, both sides are harmonic in $\Omega$, have trace $g_\kappa$ on $\Gamma_*$, and satisfy the same homogeneous
Neumann condition on $\Gamma_N$; uniqueness holds in the Poisson class fixed in
Lemma~\ref{lem:star_extension_poisson}.
The smoothing estimate \eqref{eq:poisson_l2_smoothing} from Lemma~\ref{lem:star_extension_poisson} gives
\[
\|\mathcal P_*g_\kappa\|_{C^m(K)}\le C_{K,m}\|g_\kappa\|_{L^2(\Gamma_*)}.
\]
The boundary estimate \eqref{convrategood}, with the dependence recorded in Theorem~\ref{thm:boundary_harmonic_selection}, gives the stated
rate. The $C^m_{\mathrm{loc}}$ convergence follows from this estimate. For the local $H^1$ convergence, fix
$K\cc K'\cc\Omega$ and set $w_\kappa:=\phi_\kappa-u_0$. Since $w_\kappa$ is harmonic in $K'$, the interior
Caccioppoli estimate gives; see \cite{GilbargTrudinger2001},
\[
\int_K |\nabla w_\kappa|^2\,dx
\le \frac{C}{\operatorname{dist}(K,\partial K')^2}
\int_{K'} |w_\kappa|^2\,dx .
\]
The right-hand side tends to zero by the already proved local $L^2$ convergence on $K'$, and therefore
$\phi_\kappa\to u_0$ in $H^1(K)$.
\end{proof}

We next turn from trace selection to the sharp logarithmic expansion of the energy. The proof of
Theorem~\ref{thm:energy_scale} uses the same local boundary charts, now with matching upper and lower half-annular
estimates near the $\Gamma_c$--$\Gamma_a$ junctions.

\subsection{Proof of Theorem~\ref{thm:energy_scale}}\label{subsec:proof-energy-scale}

\begin{proof}[Proof of Theorem~\ref{thm:energy_scale}]
We divide the proof in four steps.
\setcounter{proofstep}{0}

\begin{step}{Reduction and charts}
Let $N=\#\Sigma_*$. If $N=0$, the absence of points where $\Gamma_a$ and $\Gamma_c$ meet and the separation
condition allow an $H^1$ competitor whose trace on $\Gamma_a\cup\Gamma_c$ is exactly the limiting trace
$\Phi_0$. Its boundary primitive vanishes, so $J_\kappa(\phi_\kappa)=O(1)$, which is
$O(\log|\log\kappa|)$ for small $\kappa$. We therefore assume $N>0$ and put
\[
C_N:=\frac{N(\phi_c-\phi_a)^2}{2\pi}.
\]
At the $\Gamma_c$--$\Gamma_a$ junctions, choose one radius $R>0$ such that the flattening neighborhoods
$\Psi_j(B_R^+)$ are pairwise disjoint, avoid $\Gamma_N$, and have uniformly bounded $C^{1,1}$ norms and inverse
$C^{1,1}$ norms. After a rigid motion in the reference variables, each $D\Psi_j(0)$ is the identity. Shrinking $R$
if necessary, the pullback metric matrices below satisfy the uniform estimates used in the lower-bound step. All
constants are uniform in $j$ because only finitely many charts are used.
\end{step}

\begin{step}{Upper-bound competitors}
We claim that there are constants $\varepsilon_0,C>0$ such that, for every
$0<\varepsilon<\varepsilon_0$, one can find $v^\varepsilon\in H^1(\Omega)$ with
$\phi_a\le v^\varepsilon\le\phi_c$ a.e., whose trace differs from $\Phi_0$ on $\Gamma_a\cup\Gamma_c$ only inside an
$O(\varepsilon)$ neighborhood of $\Sigma_*$, and such that
\begin{equation}\label{eq:exact_upper_dirichlet}
\frac12\int_\Omega |\nabla v^\varepsilon|^2\,dx
\le C_N|\log\varepsilon|+C,
\end{equation}
\begin{equation}\label{eq:exact_upper_boundary}
\int_{\Gamma_c}I_c(v^\varepsilon)\,ds+\int_{\Gamma_a}I_a(v^\varepsilon)\,ds\le C\varepsilon .
\end{equation}

In a fixed transition-point neighborhood, write $q_+$ and $q_-$ for the values of $\Phi_0$ on the positive and
negative boundary rays. Then $|q_+-q_-|=\phi_c-\phi_a$. In polar coordinates $(\rho,\theta)$ on $B_R^+$, define
\[
P(\rho,\theta):=q_++\frac{q_--q_+}{\pi}\theta,\qquad 0<\theta<\pi.
\]
This profile is harmonic, takes the endpoint values $q_+$ and $q_-$ on the two boundary rays, and satisfies
$|\nabla P|=(\phi_c-\phi_a)/(\pi\rho)$. The corresponding flattened half-annular construction is shown in Figure~\ref{fig:half_annulus_profile}.

\begin{paperfigure}[t]
\centering
\begin{tikzpicture}[x=1.05cm,y=1.05cm,>=Latex,font=\small]
\def\r{0.82}
\def\R{2.28}

\path[fill=blue!5,draw=none,even odd rule]
  (-\R,0) arc[start angle=180,end angle=0,radius=\R] -- cycle
  (-\r,0) arc[start angle=180,end angle=0,radius=\r] -- cycle;

\draw[very thick] (-\R,0) arc[start angle=180,end angle=0,radius=\R];
\draw[very thick] (-\r,0) arc[start angle=180,end angle=0,radius=\r];
\draw[very thick,red] (-\R,0) -- (-\r,0);
\draw[very thick,blue] (\r,0) -- (\R,0);

\foreach \ang in {25,50,75,105,130,155} {
  \draw[blue!40,line width=0.55pt]
    ({\r*cos(\ang)},{\r*sin(\ang)})
    -- ({\R*cos(\ang)},{\R*sin(\ang)});
}

\draw[->,thick]
  ({(\r+0.03)*cos(32)},{(\r+0.03)*sin(32)})
  -- ({(\R-0.08)*cos(32)},{(\R-0.08)*sin(32)});
\node[anchor=south] at ({1.55*cos(32)},{1.55*sin(32)+0.18}) {$\rho$};

\draw[->,thick] (0.66,0) arc[start angle=0,end angle=72,radius=0.66];
\node[anchor=north] at ({0.66*cos(72)},{0.66*sin(72)-0.08}) {$\theta$};

\node[align=center] at (0,2.72)
  {$P(\rho,\theta)=q_+ +\dfrac{q_- - q_+}{\pi}\theta$};
\node[anchor=south east] at ({\R*cos(132)-0.05},{\R*sin(132)+0.08}) {$\rho=R$};

\node[below=7pt,red] at (-1.55,0) {$q_-$ on $\Gamma_a$};
\node[below=7pt,blue] at (1.55,0) {$q_+$ on $\Gamma_c$};
\node at (-0.18,-0.28) {$\rho=\varepsilon$};
\node at (0,-0.82) {$A_{\varepsilon,R}^+$};
\end{tikzpicture}
 \caption{Flattened half-annular profile near a smooth boundary jump. The flat boundary carries the two limiting values $q_-$ and $q_+$, and the angular harmonic interpolation on $\varepsilon<\rho<R$ gives the logarithmic Dirichlet contribution.}
\label{fig:half_annulus_profile}
\end{paperfigure}

Let $\chi_\varepsilon$ be a smooth radial cutoff with $\chi_\varepsilon=0$ on $\rho\le\varepsilon$,
$\chi_\varepsilon=1$ on $\rho\ge2\varepsilon$, and $|\chi_\varepsilon'|\le C/\varepsilon$. Put
$\bar q=(q_++q_-)/2$ and
\[
P^\varepsilon:=\bar q+\chi_\varepsilon(\rho)(P-\bar q)
\qquad\text{in }B_{R/2}^+.
\]
Then $P^\varepsilon\in H^1(B_{R/2}^+)$, stays in $[\phi_a,\phi_c]$, and differs from the limiting trace on the
flat boundary only for $|s|\le 2\varepsilon$. The energy in
$B_{2\varepsilon}^+\setminus B_\varepsilon^+$ is bounded uniformly in $\varepsilon$: indeed
$|\nabla P|\le C/\rho$, $|\chi_\varepsilon'|\le C/\varepsilon$, and $|P-\bar q|\le C$, so the two terms in
\[
\nabla P^\varepsilon=\chi_\varepsilon\nabla P+\chi_\varepsilon'(P-\bar q)e_\rho
\]
have square integrals bounded by constants over an annulus of radii comparable to $\varepsilon$. On
$B_{R/2}^+\setminus B_{2\varepsilon}^+$ one has $P^\varepsilon=P$, and therefore
\[
\frac12\int_{B_{R/2}^+\setminus B_{2\varepsilon}^+}|\nabla P^\varepsilon|^2
=\frac{(\phi_c-\phi_a)^2}{2\pi}\log\frac{R}{4\varepsilon}.
\]
On the fixed annulus $B_R^+\setminus B_{R/2}^+$, choose once and for all a connector
$W\in H^1(B_R^+\setminus B_{R/2}^+)$ whose trace equals $P(R/2,\theta)$ on the inner semicircle, equals
$q_-$ and $q_+$ on the negative and positive flat boundary segments, and is smooth on the outer semicircle. Such a
connector follows from the trace extension theorem on the fixed Lipschitz half-annulus, and its energy is independent
of $\varepsilon$.

Outside the union of the transition-point neighborhoods, remove the pulled-back half-disks $\Psi_j(B_R^+)$ and work
on the remaining fixed Lipschitz domain. The boundary data there consist of the exact constants $\phi_a,\phi_c$ on
$\Gamma_a$ and $\Gamma_c$, the fixed traces supplied on the outer semicircles by the connectors, and arbitrary fixed
bounded traces on $\Gamma_N$. These data are mutually compatible at the endpoints of each transition-point
neighborhood because the connectors preserve the flat-side values $q_\pm$. A fixed trace-extension theorem on this
fixed perforated domain gives an exterior function with $O(1)$ energy, uniformly in $\varepsilon$.

It remains to pass from the flattened half-disk to $\Omega$. If $y=(s,t)$ and $x=\Psi_j(y)$, then the Dirichlet
integral of a pullback profile has the form
\[
\int_{\Psi_j(E)} |\nabla_x v|^2\,dx
=\int_E A_j(y)\nabla_y(v\circ\Psi_j)\cdot\nabla_y(v\circ\Psi_j)\,dy,
\]
for each reference subdomain $E\subset B_R^+$, where $A_j(0)=I$ because $D\Psi_j(0)$ is orthogonal. The
$C^{1,1}$ regularity gives $|A_j(y)-I|\le C|y|$. Therefore the metric error on the principal annulus is bounded by
\[
C\int_{2\varepsilon}^{R/2}\rho\,\frac{(\phi_c-\phi_a)^2}{\rho^2}\,\rho\,d\rho\le C,
\]
and the fixed inner and outer annuli remain $O(1)$. Summing over the disjoint transition-point neighborhoods proves
\eqref{eq:exact_upper_dirichlet}. The boundary primitive is zero wherever the trace equals $\Phi_0$; its support on
the part of $\Gamma_a\cup\Gamma_c$ where the trace is modified has total length $O(\varepsilon)$, and $I_a,I_c$ are
bounded on $[\phi_a,\phi_c]$. This proves \eqref{eq:exact_upper_boundary}.

Minimality and the choice $\varepsilon=\kappa$ now give
\[
J_\kappa(\phi_\kappa)
\le C_N|\log\kappa|+C.
\]
\end{step}

\begin{step}{Lower bound for traces}
We next prove the converse estimate needed for the minimizer. Suppose $u_\kappa\in H^1(\Omega)$ satisfies
$\phi_a\le u_\kappa\le\phi_c$ a.e. and, for some fixed $B>0$,
\[
\|\operatorname{Tr}u_\kappa-\Phi_0\|_{L^2(\Gamma_*)}^2\le B\kappa|\log\kappa|.
\]
Then, for all sufficiently small $\kappa$,
\begin{equation}\label{eq:exact_lower_annuli}
\frac12\int_\Omega|\nabla u_\kappa|^2\,dx
\ge C_N|\log\kappa|-C\log|\log\kappa|.
\end{equation}

Set
\[
e_\kappa^2:=\|\operatorname{Tr}u_\kappa-\Phi_0\|_{L^2(\Gamma_*)}^2,
\qquad
r_\kappa:=\kappa|\log\kappa|^2 .
\]
Then $e_\kappa^2/r_\kappa\le B/|\log\kappa|$. Use the charts fixed in Step 1. Their boundary Jacobians are uniformly
comparable to one, so $L^2$ estimates on the boundary and in flat coordinates differ only by fixed
multiplicative constants.

Fix one transition point. On the positive and negative flat boundary rays let $q_+$ and $q_-$ be the corresponding
values of $\Phi_0$, so $|q_+-q_-|=\phi_c-\phi_a$. Define
\[
a_{\kappa,j}(\rho):=\operatorname{Tr}(u_\kappa\circ\Psi_j)(-\rho,0),\qquad
b_{\kappa,j}(\rho):=\operatorname{Tr}(u_\kappa\circ\Psi_j)(\rho,0).
\]
The global trace estimate gives the local bound
\[
E_{\kappa,j}^2:=
\int_0^R\left(|a_{\kappa,j}(\rho)-q_-|^2
+|b_{\kappa,j}(\rho)-q_+|^2\right)\,d\rho
\le C e_\kappa^2 .
\]
By the range assumption, $|u_\kappa|$ is uniformly bounded by a constant depending only on
$\phi_a,\phi_c$. Applying Lemma~\ref{lem:local_annular_lower_bound} with $r=r_\kappa$ gives
\[
\frac12\int_{\Psi_j(A_{r_\kappa,R}^+)}|\nabla u_\kappa|^2\,dx
\ge
\frac{(\phi_c-\phi_a)^2}{2\pi}\log\frac{R}{r_\kappa}
-C\left(1+E_{\kappa,j}r_\kappa^{-1/2}\right).
\]
Since $E_{\kappa,j}r_\kappa^{-1/2}\le C|\log\kappa|^{-1/2}$, the error term is bounded uniformly in $\kappa$.
Therefore
\[
\frac12\int_{\Psi_j(A_{r_\kappa,R}^+)}|\nabla u_\kappa|^2\,dx
\ge
\frac{(\phi_c-\phi_a)^2}{2\pi}\log\frac{R}{r_\kappa}-C .
\]
The annular neighborhoods are disjoint for different $j$, and
\[
\log\frac{R}{r_\kappa}=|\log\kappa|-2\log|\log\kappa|+\log R .
\]
Summing over the $N$ points in $\Sigma_*$ proves \eqref{eq:exact_lower_annuli}.
\end{step}

\begin{step}{Conclusion for the minimizer}
Theorem~\ref{thm:boundary_harmonic_selection} gives
\[
\|\operatorname{Tr}\phi_\kappa-\Phi_0\|_{L^2(\Gamma_*)}^2\le C\kappa|\log\kappa|,
\]
By the range bound from
Lemma~\ref{lem:variational_tools}(i), Step 3 applies with $u_\kappa=\phi_\kappa$. Since the boundary part of
$J_\kappa$ is nonnegative,
\[
J_\kappa(\phi_\kappa)
\ge C_N|\log\kappa|-C\log|\log\kappa|.
\]
Combining this with the upper bound from Step 2 proves \eqref{eq:exact_log_constant}.
\end{step}
\end{proof}

\begin{remark}[Mesh scale near smooth boundary jumps]\label{rem:mesh_scale}
The sharp upper-bound construction in the proof of Theorem~\ref{thm:energy_scale} also identifies the inner length
scale in the local construction. Each jump point has a small $\Gamma_*$-neighborhood disjoint from
$\overline{\Gamma_N}$, so the calculation is local. Near a smooth point where $\Gamma_a$ and $\Gamma_c$ meet, flatten
the boundary and write the two limiting trace values as $q_-$ and $q_+$. For a length scale $\delta$, use in the
half-disk the angular interpolation
\[
P(\rho,\theta)=q_++\frac{q_--q_+}{\pi}\theta,\qquad 0<\theta<\pi,
\]
on the annulus $\delta<\rho<R$, cut it off to the average $(q_-+q_+)/2$ on $\rho\le\delta$, and connect it to fixed
outer data on $\rho=R$. Repeating that construction with the inner cutoff radius $\delta$ and keeping the exact
limiting trace away from the $\delta$-neighborhood of $\Sigma_*$ gives competitors $v_\kappa^\delta$ with principal
Dirichlet contribution
\begin{equation}\label{eq:mesh_rule_dirichlet}
\frac12\int_\Omega |\nabla v_\kappa^{\delta}|^2\,dx
=\frac{N(\phi_c-\phi_a)^2}{2\pi}\log\frac{1}{\delta}+O(1),
\end{equation}
with $O(1)$ uniform in $\delta$. The boundary integral terms vanish on
$\{|s|\ge \delta\}\cap\Gamma_*$ since the trace equals $\Phi_0$ there, and on $\{|s|<\delta\}$ they are controlled by
the bounded primitives $I_a,I_c$ on $[\phi_a,\phi_c]$, giving
\begin{equation}\label{eq:mesh_rule_bv}
\frac1\kappa\!\left(\int_{\Gamma_c\cap\{|s|<\delta\}}\!\!\!I_c(v_\kappa^{\delta})\,ds
+\int_{\Gamma_a\cap\{|s|<\delta\}}\!\!\!I_a(v_\kappa^{\delta})\,ds\right)
\le \frac{C\delta}{\kappa}.
\end{equation}
This balance should be read locally. At one smooth jump the same construction gives the reduced cost
\[
F_{\kappa,j}(\delta)
:=\frac{(\phi_c-\phi_a)^2}{2\pi}\log\frac{1}{\delta}
+\frac{C_j\delta}{\kappa}+O(1),
\]
where $C_j>0$ depends on the chosen inner cutoff, the functions $I_a,I_c$, and the chart, but not on
$\kappa$ or $\delta$. This reduced cost has a stationary point satisfying
\[
\delta_{\kappa,j}^{\,*}=\frac{(\phi_c-\phi_a)^2}{2\pi C_j}\,\kappa .
\]
Thus the balance fixes the local order of the transition scale, not a universal prefactor. Summing over the finitely many
junctions gives the sharp leading logarithmic constant in Theorem~\ref{thm:energy_scale}; each smooth jump
retains the same local length scale $\delta_{\kappa,j}^{\,*}\simeq\kappa$.

For a $P^1$ discretization this provides a mesh-design heuristic: $\kappa$-robust grids should use boundary elements
of size $h_{\min}\approx\kappa$ near each smooth point where $\Gamma_a$ and $\Gamma_c$ meet and grade away from
$\Sigma_*$.
\end{remark}
 % !TEX root = ../main.tex
\section{Numerical validation}\label{sec:numerics}

We perform finite-element experiments to test the boundary convergence, interior selection, sharp energy coefficient, and mesh rule from Subsection~\ref{subsec:singular-limit-theorems}.

\subsection{Test problems and discrete setup}\label{subsec:numerics_setup}

Unless otherwise stated, the numerical values of the electrochemical parameters are those in
Table~\ref{tab:numerical_parameters}.

\begin{table}[htbp]
\centering
\caption{Baseline electrochemical parameters used in the numerical benchmarks.}
\label{tab:numerical_parameters}
\begin{tabular}{lll}
\toprule
Parameter & Value & Role \\
\midrule
$\phi_c$ & $0.2$ & cathodic equilibrium potential \\
$\phi_a$ & $-0.2$ & anodic equilibrium potential \\
$i_{c,0}$ & $3.0\times10^{-4}$ & cathodic exchange-current scale \\
$i_{a,0}$ & $3.0\times10^{-2}$ & anodic exchange-current scale \\
$C_1=A_2$ & $83.13$ & positive exponential slopes \\
$C_2=A_1$ & $35.63$ & negative exponential slopes \\
\bottomrule
\end{tabular}
\end{table}

Let $V_h$ be the conforming $P_1$ finite-element space on the corresponding triangular mesh. The computed solution
$\phi_{\kappa,h}\in V_h$ is the minimizer of
\begin{equation}\label{eq:Jkh}
J_{\kappa,h}(v_h):=
\frac12\int_\Omega |\nabla v_h|^2\,dx
+ \frac1\kappa\left(\int_{\Gamma_c} I_c(v_h)\,ds+\int_{\Gamma_a} I_a(v_h)\,ds\right),
\qquad v_h\in V_h.
\end{equation}
Strict convexity gives a unique discrete minimizer.
The mixed harmonic reference $u_{0,h}^{\mathrm{mix}}$ is computed on the same mesh as the nonlinear solution.

In the experiments, we report
\[
C_N:=\frac{N(\phi_c-\phi_a)^2}{2\pi},
\]
where $N$ is the number of smooth points where $\Gamma_a$ and $\Gamma_c$ meet in the benchmark, and
\[
\begin{aligned}
R_b(\kappa)&:=
\frac{\|\phi_{\kappa,h}-\Phi_0\|^2_{L^2(\Gamma_*)}}{\kappa|\log\kappa|},\\
R_E(\kappa)&:=
\frac{J_{\kappa,h}(\phi_{\kappa,h})}{C_N|\log\kappa|},\\
E_K(\kappa)&:=
\|\phi_{\kappa,h}-u_{0,h}^{\mathrm{mix}}\|_{L^2(K)},
\end{aligned}
\]
measure the boundary, energy, and interior diagnostics, with the compact set $K$ specified in the corresponding
benchmark.

The nonlinear systems are solved by damped Newton from $\kappa=1$ downward, with sparse direct solves and an L-BFGS-B fallback only when Newton stagnates. A point is certified if
\[
\|\nabla I(\lambda_h)\|_{\ell^\infty}\le
10^{-9}\max\{1,\|\nabla I(\lambda_h^{(0)})\|_{\ell^\infty}\},
\]
where $\lambda_h^{(0)}$ is the initial guess. The $P_1$ Dirichlet term is assembled exactly on each triangle, and the nonlinear boundary terms are evaluated by a
four-point Gauss--Legendre rule on each reactive boundary edge. Newton uses an Armijo backtracking line search with
tolerance $10^{-10}$ and step tolerance $10^{-12}$; the L-BFGS-B fallback uses the same box
$[\phi_a,\phi_c]$ and a maximum of $4000$ iterations. The experiments were run with Python~3.13.9.

\subsection{Single boundary-jump diagnostics}\label{subsec:numerics_single}

The first benchmark isolates boundary and interior convergence on the rectangle
\[
L_x=0.02,\qquad L_y=0.01,\qquad
\Omega_{\mathrm{rec}}=[0,L_x]\times[0,L_y],
\qquad
\Gamma_*=[0,L_x]\times\{0\},
\]
with homogeneous Neumann data on the other three sides. The bottom side is split into one cathodic interval and one
anodic interval,
\[
\Gamma_c=[0,L_x/2]\times\{0\},\qquad
\Gamma_a=[L_x/2,L_x]\times\{0\},
\]
up to the common endpoint, so $N=1$. For the interior diagnostic in this single-jump benchmark we use
\[
K=[0.25L_x,0.75L_x]\times[0.25L_y,0.75L_y].
\]

Figure~\ref{fig:kappa_evolution_3d} shows $\phi_{\kappa,h}$ at $\kappa=10^{-3}$, $10^{-5}$, and $10^{-7}$. The boundary transition sharpens near the jump point, while the interior field approaches the harmonic limit.

\begin{paperfigure}[htbp]
  \centering
  \includegraphics[width=\textwidth]{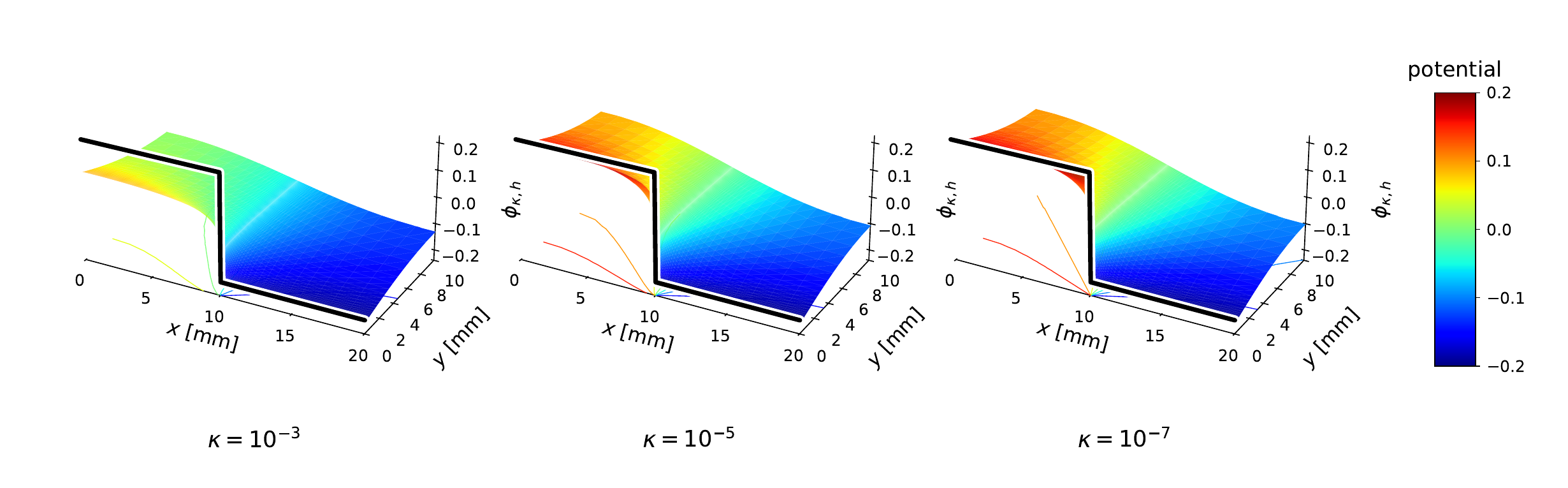}
  \caption{Three-dimensional single-jump evolution. From \textbf{left} to \textbf{right}, the finite-element potential sheet
  $\phi_{\kappa,h}(x,y)$ at $\kappa=10^{-3}$, $10^{-5}$, and $10^{-7}$, with projected level curves on the base
  plane. The black curve along $y=0$ marks the prescribed step boundary datum $\Phi_0$.}
  \label{fig:kappa_evolution_3d}
\end{paperfigure}

Figure~\ref{fig:theorem_dashboard} provides the quantitative checks. The $L^2$ boundary-error slope is $1.065$, consistent with the $s=0$ part of Theorem~\ref{thm:boundary_harmonic_selection}, namely bounded $R_b(\kappa)$. The sampled $C^0(K)$ and $C^1(K)$ slopes are $0.798$ and $0.725$ in this window; these faster observed decays are consistent with the scale from Corollary~\ref{cor:interior_selection}.

\begin{paperfigure}[htbp]
  \centering
  \includegraphics[width=1.0\textwidth]{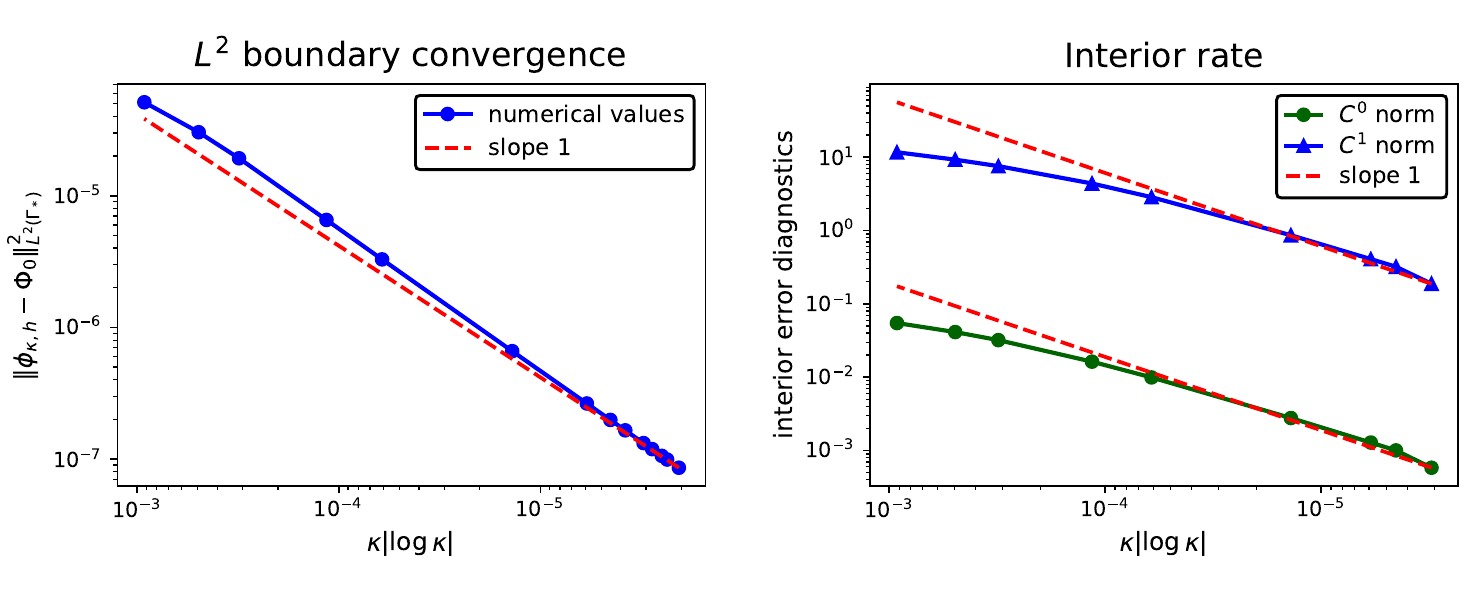}
  \caption{Single-rectangle theorem dashboard. \textbf{Left}: the $s=0$ boundary convergence diagnostic from Theorem~\ref{thm:boundary_harmonic_selection}. \textbf{Right}:
  sampled $C^0$ and $C^1$ diagnostics from Corollary~\ref{cor:interior_selection} on
  $K=[0.25L_x,0.75L_x]\times[0.25L_y,0.75L_y]$. Solid curves connect the certified numerical values; dashed lines are
  slope-one guides anchored to the data.}
  \label{fig:theorem_dashboard}
\end{paperfigure}

\subsection{Multiple boundary-jump diagnostics}\label{subsec:numerics_multi_corrugated}

These tests check whether the normalized leading coefficient is local in the number of smooth jump points and
insensitive to moderate changes in the zero-flux part of the geometry. The rectangular multi-jump benchmark uses the
same rectangle $\Omega_{\mathrm{rec}}$ and bottom side $\Gamma_*$ as the single-jump test, but with four alternating
intervals of equal length,
\[
[0,L_x/4],\quad [L_x/4,L_x/2],\quad [L_x/2,3L_x/4],\quad [3L_x/4,L_x],
\]
so that there are three interior cathode-anode jump points. The nonrectangular tests use the corrugated-top domain
\begin{align*}
\Omega_{\mathrm{cor}}&=\{(x,y):0<x<L_x,\ 0<y<H(x)\},\\
H(x)&=L_y\left(1+0.22\sin\frac{2\pi x}{L_x}
+0.10\sin\left(\frac{4\pi x}{L_x}+0.5\right)\right),
\end{align*}
with reactions imposed on the bottom side and homogeneous Neumann data on the remaining boundary. The corrugated
$N=4$ and $N=6$ tests use alternating $\Gamma_c$ and $\Gamma_a$ intervals on the bottom side, with jump locations
\[
0.18L_x,\ 0.39L_x,\ 0.60L_x,\ 0.81L_x
\]
and
\[
0.12L_x,\ 0.27L_x,\ 0.40L_x,\ 0.55L_x,\ 0.70L_x,\ 0.84L_x,
\]
respectively.
The corrugated domains are included as nonrectangular curvilinear-polygonal tests within the geometric setting of the
theory; the cathode-anode junctions themselves remain smooth flat-boundary points, separated from the zero-flux
boundary, so they preserve the local structure that determines the logarithmic coefficient.

The corrugated experiment with four jump points is closer to the limiting regime, reaching
$R_E=0.938$ at $\kappa=5\times10^{-6}$. Figure~\ref{fig:corrugated_six_diagnostics} gives the corrugated $N=6$ run;
its certified ratio reaches $0.943$ at $\kappa=3\times10^{-7}$.

\begin{paperfigure}[htbp]
  \centering
  \includegraphics[width=1.0\textwidth]{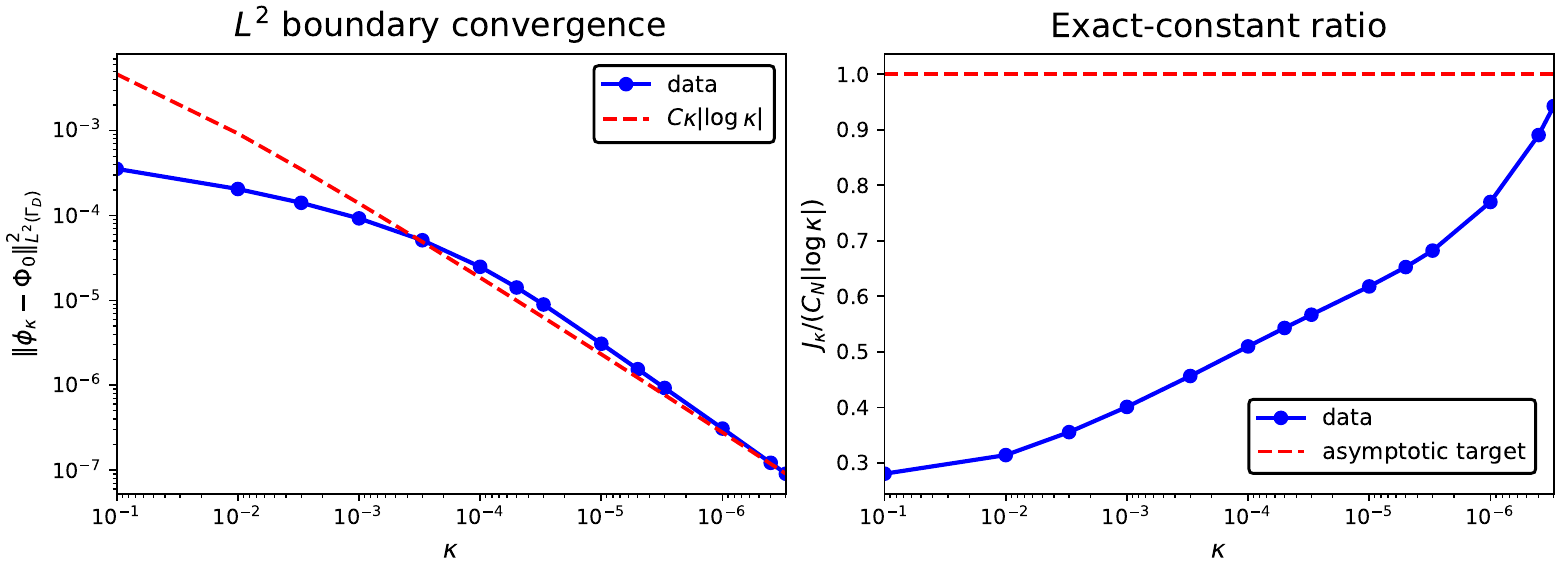}
  \caption{Corrugated six-jump diagnostics. \textbf{Left}: $L^2$ boundary convergence to $\Phi_0$ over the admissible data. \textbf{Right}: the normalized exact-constant ratio in the theorem-facing window. The retained points show the finite-window approach of $R_E(\kappa)$ toward the predicted unit limit while excluding smaller runs that fail the residual or range certificate.}
  \label{fig:corrugated_six_diagnostics}
\end{paperfigure}

Figure~\ref{fig:practical_stainless_zro2_diagnostics} repeats the corrugated tests with parameter ratios from a ZrO$_2$-coated stainless-steel study \cite{PengJiShiWang2021}. The retained data reach $R_E=0.978$ for $N=4$ and $R_E=0.950$ for $N=6$ at $\kappa=5\times10^{-7}$; the next point is excluded by the residual certificate. These parameter-variant runs are therefore used only as robustness checks.

\begin{paperfigure}[htbp]
  \centering
  \includegraphics[width=1.0\textwidth]{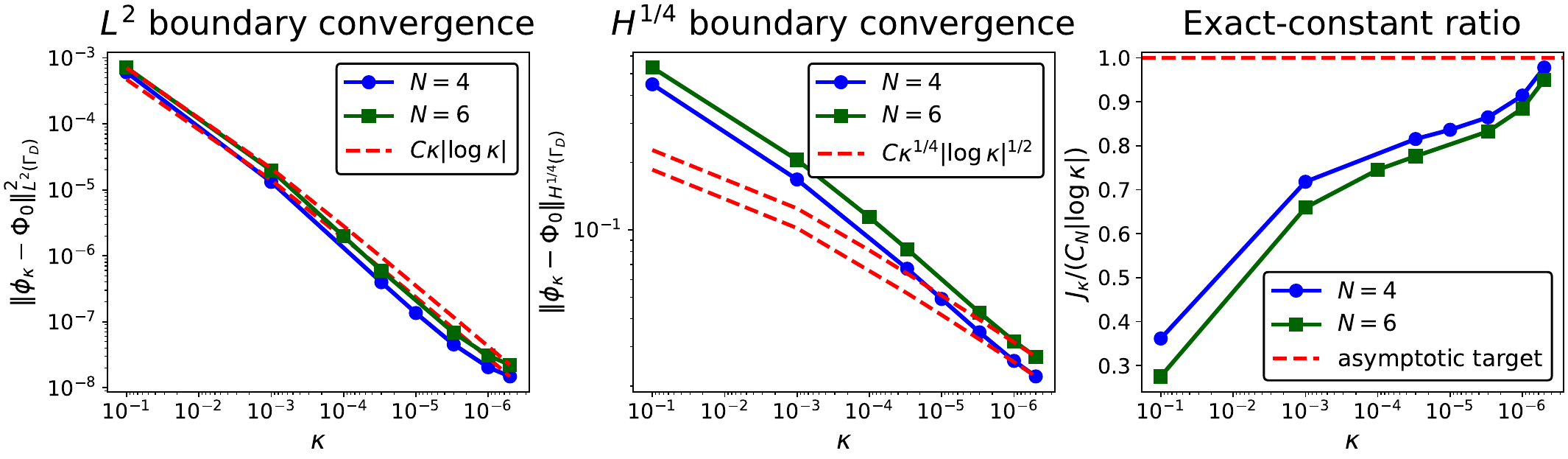}
  \caption{Corrugated four- and six-jump diagnostics with one-layer ZrO$_2$-coated stainless-steel
  parameter ratios from \cite{PengJiShiWang2021}. The first two panels compare $L^2$ and $H^{1/4}$ boundary convergence for $N=4$ and $N=6$; the right panel shows that the energy ratios remain close to the same exact-constant normalization over the range.}
  \label{fig:practical_stainless_zro2_diagnostics}
\end{paperfigure}

\subsection{Mesh rule}\label{subsec:numerics_mesh}

The mesh rule suggests $h_{\min}\simeq\kappa$ near each smooth jump point. We use the single-jump rectangle for the
quantitative mesh-comparison errors and the corrugated four-jump geometry for the mesh visualization.

For each target conductivity we compare three meshes:
\begin{itemize}
\item \textbf{Uniform ($N$ nodes)}: uniform mesh with the same node count as the graded mesh.
\item \textbf{Uniform ($4N$ nodes)}: uniform mesh with roughly four times as many nodes.
\item \textbf{Graded ($N$ nodes)}: local point refinement centered at the boundary jump, with
$h_{\min}=0.25\,\kappa$.
\end{itemize}
More precisely, if $z_j=(x_j,0)$ is a boundary jump point, the numerical graded mesh adds the points
\[
z_{j,m,\ell}=(x_j+d_m\cos\theta_\ell,\ d_m\sin\theta_\ell),
\qquad
d_m=h_{\min}m^2,\qquad
\theta_\ell=\frac{\ell\pi}{16},
\]
for all $m\ge1$ such that $d_m\le R_{\rm num}=4.0\times10^{-3}\,\mathrm{m}$ and for
$\ell=0,\ldots,16$, retaining only points in the domain. These points are added to a coarse background grid and
then triangulated.

Figure~\ref{fig:mesh_strategy_triptych} shows the same corrugated four-junction geometry on a coarser mesh, so that the
grading pattern is visible. Both displayed meshes use $N=435$ nodes. The left panel uses uniform $x$- and
$y$-coordinates, while the right panel uses local pointwise refinement around the boundary junction points
$z_j=(x_j,0)$. With $R=1.6\times10^{-3}\,\mathrm{m}$, $M=7$, $L=11$, and $q=2$, set
\[
x_j\in\{0.0036,\ 0.0078,\ 0.0120,\ 0.0162\}\,\mathrm{m},
\qquad
d_m=R\left(\frac{m}{M}\right)^q,\qquad
\theta_\ell=\frac{\ell\pi}{L}.
\]
The local refinement points are
\[
z_{j,m,\ell}=(x_j+d_m\cos\theta_\ell,\ d_m\sin\theta_\ell),
\qquad m=1,\ldots,M,\quad \ell=0,\ldots,L,
\]
retaining only points in $(0,L_x)\times[0,L_y]$. They are added to a coarse background grid and triangulated in the
reference rectangle. For the corrugated domain, the resulting reference mesh is mapped by
\[
(x,y)\longmapsto \left(x,\frac{y}{L_y}H(x)\right),
\]
with $H$ as in Subsection~\ref{subsec:numerics_multi_corrugated}. This places the smallest elements near the points
$(x_j,0)$, not along the full vertical lines $x=x_j$.

\begin{paperfigure}[htbp]
  \centering
  \includegraphics[width=0.92\textwidth]{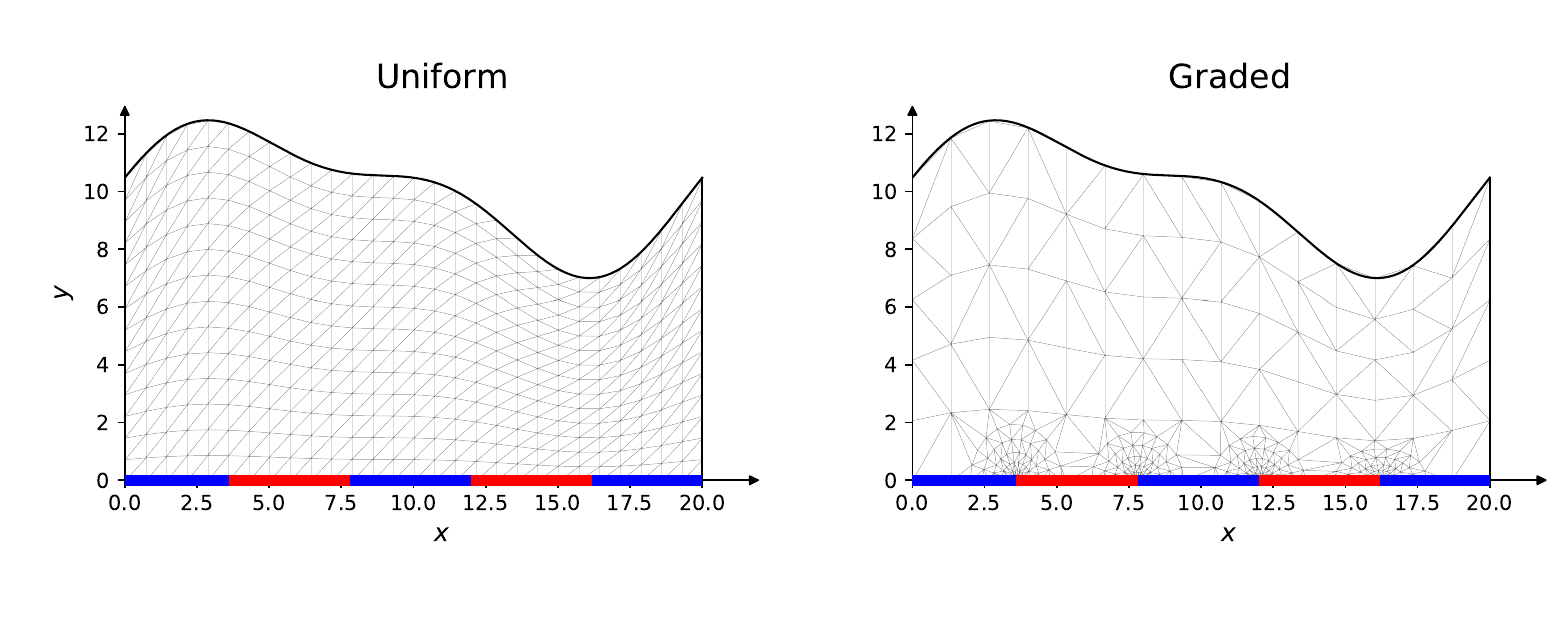}
  \caption{Coarsened mesh visualization for the corrugated four-junction geometry. The blue and red boundary segments mark
  the two prescribed values of $\Phi_0$; the graded mesh concentrates elements near the cathode-anode junctions.}
  \label{fig:mesh_strategy_triptych}
\end{paperfigure}

In the single-jump rectangle, errors are measured against a reference solution
$\phi_{\kappa,h_*}$ computed on the same problem with $h_{\min}=0.05\,\kappa$ at the
boundary jump.

\begin{paperfigure}[htbp]
  \centering
  \includegraphics[width=1.0\textwidth]{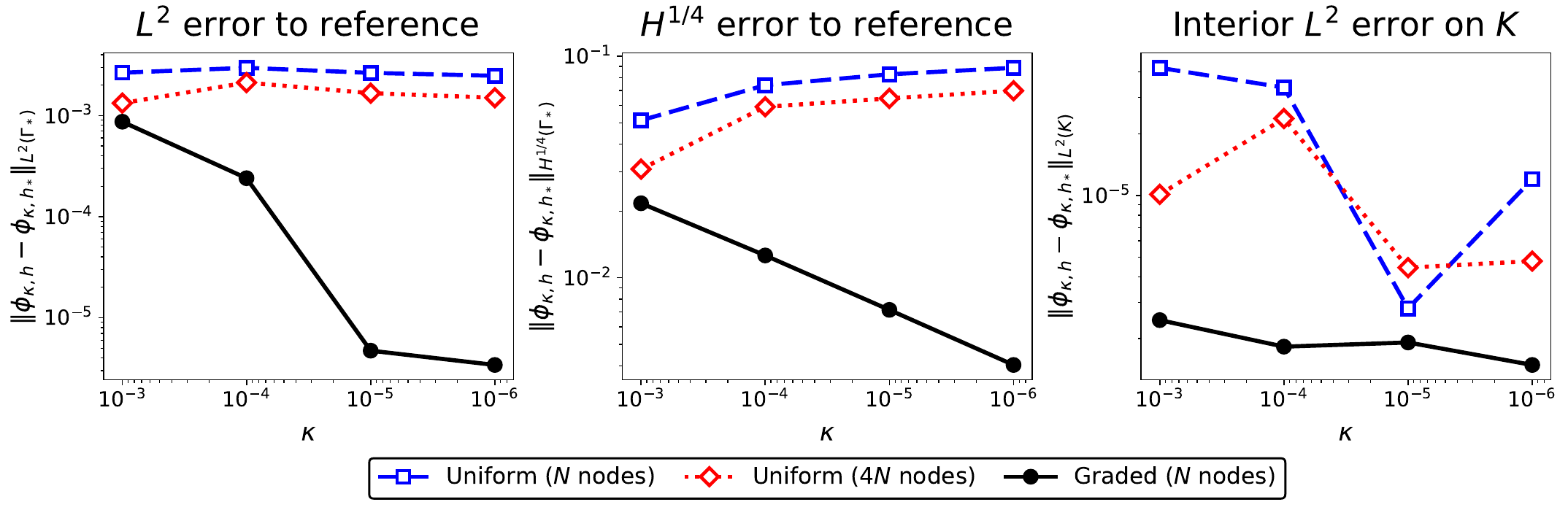}
  \caption{Single-jump mesh comparison. \textbf{Left}: boundary $L^2$ error
  $\|\phi_{\kappa,h}-\phi_{\kappa,h_*}\|_{L^2(\Gamma_*)}$ against an over-resolved
  graded reference. \textbf{Middle}: boundary $H^{1/4}$ error against the same reference. \textbf{Right}: interior $L^2$ error on the inset rectangle $K$. The
  graded mesh (black, solid) matches or beats the $4N$-uniform mesh
  (red, dotted) on all three diagnostics while using the same node budget as the
  $N$-uniform mesh (blue, dashed).}
  \label{fig:mesh_comparison_refined}
\end{paperfigure}

Figure~\ref{fig:mesh_comparison_refined} shows the payoff of grading. At $\kappa=10^{-5}$ the graded mesh uses $930$ nodes and reaches boundary error $4.7\times10^{-6}$, while the $4N$ uniform mesh uses $3{,}655$ nodes and remains near $1.7\times10^{-3}$. The energy ratio behaves similarly: it stays in the predicted band $0.59$--$0.77$ on graded meshes but already rises to about $18$ and $9$ on the two uniform meshes at $\kappa=10^{-5}$.
 % !TEX root = ../main.tex
\section{Conclusions and future directions}\label{sec:conclusion}

This work shows that the low-conductivity singular limit is controlled by a boundary layer localized at points where
$\Gamma_a$ and $\Gamma_c$ meet. The stiff boundary reaction drives the trace on $\Gamma_a\cup\Gamma_c$ toward the piecewise equilibrium state, and the
interior potential is then selected by the mixed harmonic field associated with that limiting boundary datum. The central
point is that the leading energy growth is not a global feature of the whole domain geometry, but a local
logarithmic contribution from each smooth boundary jump.

The supporting analysis separates what is standard from what is singular. Solvability and regularity follow from
trace-critical variational estimates and classical elliptic theory, while the singular behavior comes from the
borderline cost of smoothing a boundary jump.

The numerical experiments are designed around this distinction. They
display the formation of the trace layer, compare the nonlinear solutions with the harmonic limit, test the boundary
and interior convergence diagnostics, and show the same mechanism
in a corrugated geometry with several boundary jumps.

We close by highlighting several directions in which the present analysis can be extended, and which we plan to
investigate in future work.

\begin{itemize}
\item \textbf{Variable conductivity.} Spatially varying conductivity would couple the local jump energy to an inhomogeneous
interior metric and could change both the local cell problem and the harmonic selection principle. The most interesting
regime is one in which the conductivity varies on comparable scales to the reaction layer, since then homogenization
and jump asymptotics would no longer decouple cleanly.

\item \textbf{Singular perturbations of the domain and moving boundaries.} Since the leading constant is local at the
cathode-anode junctions, it should be stable under smooth perturbations of the domain and of the reaction partition.
A quantitative version of this stability would be useful for corrosion models in which the boundary evolves and the
elliptic problem is solved on a slowly varying sequence of domains \cite{VogeliusXu1998,KavianVogelius2003}.

\item \textbf{Mixed and nonsmooth geometries.} It would be useful to remove or weaken the separation assumption in the
mixed problem, where points of $\overline{\Gamma_a}\cap\overline{\Gamma_c}$ are kept away from changes of boundary
type. A related extension is to allow the cathode-anode junction itself to lie at a corner. In a wedge of opening
angle $\alpha$, the local angular profile suggests the leading contribution
\[
\frac{(\phi_c-\phi_a)^2}{2\alpha}|\log\kappa|
\]
per corner junction, replacing the smooth-boundary denominator $2\pi$. Thus nonsmooth junctions should be governed by
the same boundary-layer mechanism, but with the local angle entering the exact constant.
\end{itemize}
 
\section*{Acknowledgments}
The author acknowledges funding from Schaeffler AG \& Co. KG. The author thanks Enrique Zuazua for insightful discussions.

\clearpage
\appendix
% !TEX root = ../main.tex
\section{Modeling reduction and nondimensional conductivity}\label{app:modeling}

This appendix records the reduced electrochemical scaling behind \eqref{eq:intro_model_pde}. It is not used in the proofs, but it
fixes the physical meaning of the small parameter.

The stationary reduction assumes constant conductivity, electroneutral transport, negligible concentration
polarization, and a fixed boundary partition, as in activation-controlled potential models for galvanic corrosion
\cite{NewmanThomasAlyea2004,Rubinstein1990}.

Let $\Phi_{\rm phys}$ be the dimensional electrolyte potential. In a well-mixed electrolyte with electroneutrality in the electrolyte interior
and negligible concentration polarization, Ohm's law gives the current density
\[
j=-\sigma \nabla \Phi_{\rm phys},
\]
where $\sigma>0$ is the electrolyte conductivity. Charge conservation gives
\[
\operatorname{div} j=0,
\qquad\text{hence}\qquad
-\Delta \Phi_{\rm phys}=0
\]
when $\sigma$ is constant. On an insulating boundary component, $j\cdot\nu=0$. On a reactive metal-electrolyte
interface, the normal current is modeled by a signed exponential current-overpotential law
\[
j\cdot\nu=i_{*,0}\left[
\exp\!\left(\frac{(1-\alpha)zF(\Phi_{\rm phys}-E_*)}{RT}\right)
-\exp\!\left(-\frac{\alpha zF(\Phi_{\rm phys}-E_*)}{RT}\right)
\right],
\]
with the equilibrium potential $E_*$ and exchange current $i_{*,0}$ chosen according to the reactive process. Here
$\alpha\in(0,1)$ is the transfer coefficient, $z$ is the charge number, $F$ is Faraday's
constant, $R$ is the gas constant, and $T$ is the absolute temperature.

Choose a length scale $L$ and a potential scale $V$. Write
\[
x=L\widehat x,\qquad
\Phi_{\rm phys}=V\phi .
\]
Then the boundary condition becomes
\[
\partial_{\widehat\nu}\phi
=-\frac{L i_{*,0}}{\sigma V}\,\widehat i_*(\phi),
\]
up to the sign convention for the outward normal and with $\widehat i_*$ denoting the corresponding dimensionless
signed exponential current. After absorbing fixed exchange-current ratios and thermal-voltage constants into the
dimensionless currents $i_a$ and $i_c$, the common conductivity parameter used in the paper is
\[
\kappa:=\frac{\sigma V}{L i_{\rm ref}},
\]
where $i_{\rm ref}$ is the reference exchange-current scale. Thus the low-conductivity regime
$\kappa\to0^+$ corresponds either to decreasing electrolyte conductivity, increasing device length, or increasing
interfacial exchange-current scale relative to Ohmic transport through the electrolyte interior. In this regime the dimensionless model takes the
form
\[
\partial_\nu\phi=-\frac{i_a(\phi)}{\kappa}\quad\text{on }\Gamma_a,
\qquad
\partial_\nu\phi=-\frac{i_c(\phi)}{\kappa}\quad\text{on }\Gamma_c,
\]
which is \eqref{eq:intro_model_pde}. The mathematical analysis in the paper treats this reduced stationary problem; concentration
dynamics, space-charge layers, and moving-interface effects are outside the model and are listed as extensions in
Section~\ref{sec:conclusion}.
 
\end{document}